\documentclass[12pt,reqno]{amsart}

\usepackage[utf8]{inputenc}
\usepackage{amsfonts, color,epsf}
\usepackage{amsthm} 
\usepackage{amssymb} 
\usepackage{xspace}
\usepackage{xcolor}
\usepackage[margin=1in]{geometry} 
\usepackage{bm}
\usepackage{boxedminipage} 
\usepackage{enumerate}
\usepackage{comment} 
\usepackage{todonotes}
\usepackage{stmaryrd} 
\usepackage{float}
\usepackage{enumitem}
\usepackage{cancel}
\usepackage{tikz}
    \usetikzlibrary[backgrounds, fit, intersections, calc, decorations.markings]

\usepackage{pgfplots}
    \pgfplotsset{compat=1.6}

\usepackage[pdftex,bookmarks=true,pdfstartview=FitH,colorlinks,linkcolor=blue,filecolor=blue,citecolor=blue,urlcolor=blue,pagebackref=true]{hyperref}
\usepackage[lambda,advantage,operators,sets,adversary,landau,probability,notions,logic,ff,mm,primitives,events,complexity,asymptotics,keys]{cryptocode}

\newlength{\oldparindent}
\theoremstyle{plain}
    \newtheorem{theorem}{Theorem}

    \newtheorem{lemma}{Lemma}[section]
    \newtheorem{proposition}[lemma]{Proposition}

\theoremstyle{definition}
    
    \newtheorem{claim}[lemma]{Claim}

\providecommand{\namedref}[2]{%
  \expandafter\hyperref\expandafter[#2]{#1~\ref*{#2}}\xspace%
}
\providecommand{\theoremref}[1]{\namedref{Theorem}{thm:#1}}
\providecommand{\propref}[1]{\namedref{Proposition}{prop:#1}}

\providecommand{\claimref}[1]{\namedref{Claim}{clm:#1}}
\providecommand{\figureref}[1]{\namedref{Figure}{fig:#1}}
\providecommand{\sectionref}[1]{\namedref{Section}{sec:#1}}

\providecommand{\equationref}[1]{\namedref{Equation}{eq:#1}}

\providecommand{\cN}{\ensuremath{\mathcal N}\xspace}

\providecommand{\ie}[0]{\text{i.e.}\xspace}

\let\leq\leqslant
\let\geq\geqslant

\providecommand{\defeq}[0]{\ensuremath{\;\coloneqq\;}\xspace}
\providecommand{\eqdef}[0]{\ensuremath{\;\eqqcolon\;}\xspace}

\providecommand{\p}[1]{\ensuremath{^{\left(#1\right)}}\xspace}

\let\phi\varphi

\usepackage{lineno}
\providecommand{\ba}[0]{\ensuremath{\mathbf a}\xspace}
\providecommand{\bu}[0]{\ensuremath{\mathbf u}\xspace}
\providecommand{\bw}[0]{\ensuremath{\mathbf w}\xspace}
\providecommand{\btheta}[0]{\ensuremath{\bm{\theta}}\xspace}
\providecommand{\bTheta}[0]{\ensuremath{\bm{\Theta}}\xspace}

\providecommand{\trace}[0]{\ensuremath{\mathrm{tr}}\xspace}

\providecommand{\dd}[0]{\ensuremath{\mathrm d}\xspace}
\providecommand{\close}[0]{\ensuremath{\mathrm{close}}\xspace}
\providecommand{\far}[0]{\ensuremath{\mathrm{far}}\xspace}
\providecommand{\diag}[0]{\ensuremath{\mathrm{diag}}\xspace}
\providecommand{\wt}[0]{\ensuremath{\mathrm{wt}}\xspace}

\providecommand{\mye}[0]{\ensuremath{\mathrm{e}}\xspace}

\providecommand{\gelfond}{Gel'fond\xspace}

\providecommand{\mybigO}[2]{\ensuremath{\mathcal O_{#1}\left({#2}\right)}\xspace}

\let\originalleft\left
\let\originalright\right
\renewcommand{\left}{\mathopen{}\mathclose\bgroup\originalleft}
\renewcommand{\right}{\aftergroup\egroup\originalright}
\makeatletter
\renewcommand\paragraph{
  \@startsection{paragraph}{4}{0pt}%
  {1ex plus 0.25ex minus .1ex}% 1st curly brace: space before
  {-1em}%                      2nd curly brace: space after
  {\normalfont\normalsize\bfseries}
}
\makeatother

\renewcommand*\backref[1]{\ifx#1\relax \else $\uparrow$#1 \fi}

\title[Energy Estimation of the Hamming Slice]{Energy Estimation of the Hamming Slice and its Applications}

\author[A. Biswas]{Aniruddha Biswas}
\address{Department of Computer Science, Purdue University, West Lafayette, IN 47907, USA}
\email{biswas62@purdue.edu}

\author[J. Hwang]{Jihun Hwang}
\address{Department of Computer Science, Purdue University, West Lafayette, IN 47907, USA}
\email{hwang102@purdue.edu}

\author[H. K. Maji]{Hemanta K. Maji}
\address{Department of Computer Science, Purdue University, West Lafayette, IN 47907, USA}
\email{hmaji@purdue.edu}

\author[I. D. Shkredov]{Ilya D. Shkredov}
\address{Department of Mathematics, Purdue University, West Lafayette, IN 47907, USA}
\email{ishkredo@purdue.edu}

\author[X. Ye]{Xiuyu Ye}
\address{Department of Computer Science, Purdue University, West Lafayette, IN 47907, USA}
\email{ye151@purdue.edu}

\date{\today}

\begin{document}

\keywords{Additive energy, Hamming slices, sum-of-digits function, finite cyclic groups, sumsets, ternary representations, cyclic carry automata, transfer matrices}
\subjclass[2020]{Primary 11B30; Secondary 11A63, 11B13, 11B34}

\maketitle

\begin{abstract}
Let $R=\ZZ/(2^n-1)\ZZ$, where $n\geq 3$, and let $S_w\subseteq R$ be the residues whose canonical $n$-digit binary expansion has Hamming weight $w$. 
We obtain, in particular, an asymptotic formula for the additive energy of $S_w$
\[
  E(S_w)=\frac{\abs{S_w}^4}{|R|}+ \bigO{\abs R^3n^{-3}},
\]
which holds uniformly in $w$.
The error term is optimal in order, with a matching lower bound for $w=\floor{ n/2+\sqrt{n} }$.
It follows that triple sums of arbitrary unit dilates have asymptotically uniform representation counts when $\prod_{j=1}^{3} \abs{S_{w_j}} /\left(\abs R n^{-3/5}\right)^3\to\infty$,
and that double sums have asymptotically full support when $\abs{S_{w_1}} \abs{S_{w_2}}/\left(\abs Rn^{-3/4}\right)^2\to\infty$.
In the proof, modular collisions are represented using a cyclic binary carry automaton; this appears to be a novel approach in this area of problems.
\end{abstract}

\tableofcontents

\section{Introduction}
\label{sec:intro}

How much additive structure does a fixed binary digit sum force? 
The constraint is combinatorial; its additive consequences depend on the ambient group.

The arithmetic study of digit restrictions goes back to \gelfond~\cite{Gelfond1968}.
Mauduit and S\'ark\"ozy~\cite{MauduitSarkozy1997} and Mauduit, Pomerance, and S\'ark\"ozy~\cite{MauduitPomeranceSarkozy2005} studied the distribution of fixed-digit-sum sets in residue classes.
Related directions include missing digits~\cite{Konyagin2001}, digit restrictions on primes~\cite{MauduitRivat2010,Maynard2019}, and Waring's problem with restricted digits~\cite{Green2025}.
Our focus is the additive structure of fixed-digit-sum sets.
Rather than constrain the digit sum of the \emph{output} of addition, as in~\cite{MauduitRivatSarkozy2017}, we prescribe the digit sums of the \emph{inputs} and count additive collisions and representations
in a cyclic group.

Cryptography provides a further motivation.
A standard defense splits a secret into shares chosen uniformly subject to their sum equaling the secret modulo a fixed integer.
Any proper subset of the shares reveals nothing about the secret, yet a device's power consumption may expose the Hamming weight of every share~\cite{EC:FMMOS24}.
The information revealed by these side-channel observations depends on how uniformly sums from the corresponding Hamming slices represent the possible secrets.

In this paper we study the underlying additive questions in
\[
     N=2^n-1,\qquad R=\ZZ/N\ZZ,
\]
where $n\geq3$. 
Neither $n$ nor $N$ is assumed to be prime. 
We identify each residue with its representative in $\{0,\ldots,2^n-2\}$ and put
\[
     W=\{0,\ldots,n-1\},\qquad
     S_w=\{x\in R \colon \wt(x)=w\},\qquad 
     \abs{S_w}=\binom nw,
\]
where $\wt(x)$ is the number of ones in the canonical $n$-digit binary expansion of $x$. 
The modulus preserves an important part of the digit structure: multiplication by $2$ cyclically rotates the digits, so $2S_w=S_w$. 
Addition behaves differently. 
Carries propagate between digit positions, including from the last position back to the first.

We measure the additive structure of $S_w$ through its \emph{additive energy}.
For $A\subseteq R$, write
\[
    E(A)=\abs{ \vphantom{2^{2^2}} \left\{(a,b,c,d)\in A^4 \colon a+b=c+d\right\}} .
\]
Cauchy--Schwarz gives $E(A)\geq |A|^4/N$; the difference measures the mean-square deviation of the two-term representation function from its average. 
Our principal result determines the optimal order of an upper bound for this difference that is uniform over all Hamming slices. 
For central slices, the energy is within a relative $\bigO{n^{-1}}$ of the lower bound, and this order cannot be improved uniformly in the weight. 
The same estimate yields asymptotically uniform representation counts for three arbitrary unit dilates and asymptotically full support for two.

\subsection{Main results}
\label{sec:main-results}
For subsets $A_1,\ldots,A_k\subseteq R$, let
\[
    r_{A_1+\cdots+A_k}(x)
        =\abs{ \left\{ (a_1,\ldots,a_k)\in A_1\times\cdots\times A_k \colon
                  a_1+\cdots+a_k=x\right\} \vphantom{2^{2^2}}},
\]
and write $\supp{A_1+\cdots+A_k}$ for the set on which this function is positive. 
A unit dilate is $aS_w=\{ax \colon x\in S_w\}$ with $a\in R^\times$. 
All implicit constants below are absolute unless a subscript indicates otherwise. 
Limits are taken as $n\to\infty$ through integers, with the weights and units allowed to vary with $n$.

The ternary representation problem asks whether each $x\in R$ can be written as $a_1s_1+a_2s_2+a_3s_3$, with $s_j\in S_{w_j}$. 
Our first theorem counts these representations uniformly in $x$.

%%% Theorem
\begin{theorem}[Nearly-uniform representation counts]\label{thm:uniform-rep}
    For $a_1,a_2,a_3\in R^\times$, $w_1,w_2,w_3\in W$, and every $x\in R$,
    \[
        \abs{ r_{a_1\cdot S_{w_1} + a_2\cdot S_{w_2} + a_3\cdot S_{w_3}}(x) 
                - \frac{\abs{S_{w_1}}\abs{S_{w_2}}\abs{S_{w_3}} }N } 
            \ll \left(\frac {N}n\right)^{3/2} \cdot 
            \left( {\abs{S_{w_1}}  \abs{S_{w_2}}  \abs{S_{w_3}}} \right)^{1/6}.
    \]
    In particular, if
    \[
        \frac{ {\abs{S_{w_1}}  \abs{S_{w_2}}  \abs{S_{w_3}}} }{ \left(N\cdot n^{-3/5}\right)^3 } \longrightarrow\infty,
    \]
    then, uniformly in $x$,
    \[
        r_{a_1\cdot S_{w_1} + a_2\cdot S_{w_2} + a_3\cdot S_{w_3}}(x) = \left(1+\smallO1\right) \cdot \frac{\abs{S_{w_1}}\abs{S_{w_2}}\abs{S_{w_3}} }N. 
    \]
\end{theorem}
%%% 

For two summands, we bound the number of residues with no representation.
The weights in both results need not coincide, and the dilations are arbitrary units, not only powers of $2$.

%%% Theorem
\begin{theorem}[Near-full support]\label{thm:full-supp}
    For $a_1,a_2\in R^\times$ and $w_1,w_2\in W$,
    \[
        \abs{ R \setminus \supp{a_1\cdot S_{w_1} + a_2\cdot S_{w_2} } \vphantom{2^{2^2}}} 
            \ll N\cdot\min\left\{ 1 , \frac{N^4n^{-3}}{\abs{S_{w_1}}^2\abs{S_{w_2}}^2 }\right\}.
    \]
    Consequently, if
    \[
        \frac{ \abs{S_{w_1} }\abs{ S_{w_2}} }{ \left(N \cdot n^{-3/4} \right)^2 } \longrightarrow\infty,
    \]
    then
    \[
        \abs{ \supp{a_1\cdot S_{w_1} + a_2\cdot S_{w_2} } \vphantom{2^{2^2}}} = \left(1-\smallO1\right)\cdot N.
    \]
\end{theorem}
%%%

The common input is a sharp fourth-moment estimate. 
To state it, set $\mye_N(x)=\exp\left(2\pi \imath x/N\right)$ and use the normalization
    \[
         \widehat f(t)=\frac1N\sum_{x\in R}f(x)\mye_N(-xt),\qquad
         \widehat S_w=\widehat{1_{S_w}},\qquad
         \delta_w=\frac{\abs{S_w}}{N}.
    \]
Fourier inversion and Parseval give
%%%
\begin{equation}\label{eq:energy-identity}
    E(S_w)=\frac{\abs{S_w}^4}{N}
            + N^3\sum_{t\neq 0}\abs{\widehat S_w(t)}^4.
\end{equation}
%%%
Thus it is the nonzero fourth moment, rather than the full energy, that must be estimated to control the excess over the uniform baseline.

%%% Theorem
\begin{theorem}[Energy estimate]\label{thm:energy}
    For every $w\in W$,
    \[
        \sum_{t\neq 0} \abs{ \widehat{S_w}(t) } ^4  \ll n^{-3} .
    \]
    Equivalently,
    \[
        E(S_w) = \frac1N \binom nw^4 +\bigO{N^3n^{-3}}.
    \]
\end{theorem}
%%%

%%% Theorem
\begin{theorem}[Optimality of the energy error]\label{thm:lower}
    For $n\geq5$, put $w_n=\floor {n/2+\sqrt n}$. 
    Then
    \[
     \sum_{t\neq 0} \abs{\widehat{S_{w_n}}(t)}^4 \gg n^{-3} .
    \]
    In particular,
    \[
        E(S_{w_n})-\frac{\abs{S_{w_n}}^4}{N} \asymp N^3n^{-3}.
    \]
\end{theorem}
%%%
The same argument applies to any sequence of integer weights satisfying
\[
  \frac n2+a\sqrt n \leq w_n \leq \frac n2+b\sqrt n,
\]
where $0<a<b$ are fixed, and yields, uniformly in this range,
\[
  E(S_{w_n})-\frac{\abs{S_{w_n}}^4}{N}
  \asymp_{a,b} N^3n^{-3}.
\]
For simplicity, we present the proof in this draft only for the special case stated in \theoremref{lower}.

The optimality assertion concerns the \emph{uniform error term}: it does not assert a matching lower bound for every weight. 
Likewise, the error bound in \theoremref{energy} is uniform in $w$, but it need not be smaller than the main term for very small slices.

The central window makes the conclusions particularly transparent. 
For fixed $C>0$ and $\abs{w-n/2}\leq C\sqrt n$, the standard binomial estimates give $\abs{S_w}\asymp_C Nn^{-1/2}$. 
Hence
    \[
        E(S_w)=\frac{\abs{S_w}^4}{N} \cdot \left(1+\mybigO C{n^{-1}}\right).
    \]
If all the relevant weights lie in this window, then
    \[
     \begin{aligned}
         r_{a_1S_{w_1}+a_2S_{w_2}+a_3S_{w_3}}(x)
            &=\frac{\abs{S_{w_1}}\abs{S_{w_2}}\abs{S_{w_3}}}{N} \cdot
                          \left(1+\mybigO C{n^{-1/4}} \right),\\
         \abs{ R\setminus\supp{a_1S_{w_1}+a_2S_{w_2}} }
            &\ll_C N/n.
     \end{aligned}
    \]
Thus three central slices represent every residue for all sufficiently large $n$, with an asymptotic count for each residue, whereas two cover all but an $\mybigO C{n^{-1}}$ proportion of $R$. 
These conclusions have rather different flavors: near-full binary support alone does not imply the pointwise ternary asymptotic.

\subsection{Symmetry, Dissociation, and the Fourth Moment}
\label{sec:fourier-context}
The identity \(2S_w=S_w\) makes multiplicative symmetry a natural first avenue to consider. 
Indeed,
    \[
        \widehat{S_w}(t)=\widehat{S_w}(2t),
    \]
so each Fourier coefficient is repeated along an orbit of $\Gamma=\langle2\rangle=\{2^j \colon 0\leq j<n\}$. 
This places the slices alongside multiplicative subgroups and sets invariant under their action in the study
of exponential sums~\cite{KonyaginShparlinski1999,AlonBourgain2014}. 
The relevant features are the number of symmetries and the additive independence among them.

For a unit frequency $t\in R^\times$, the orbit $t\Gamma$ has $n$ elements.
Parseval therefore gives
    \[
         n \, \abs{\widehat{S_w}(t)}^2\leq\delta_w,
         \qquad
         \abs{\widehat{S_w}(t)}\leq\sqrt{\delta_w/n}.
    \]
Dissociation improves this estimate. 
Recall that a set $\Lambda$ is \emph{dissociated} if a relation $\sum_{\lambda\in\Lambda}\varepsilon_\lambda\lambda=0$, with $\varepsilon_\lambda\in\{-1,0,1\}$, forces every coefficient to vanish \cite[Definition 4.32, pp. 175--176]{TaoVu2006}. 
Although $\Gamma$ itself has the relation $1+2+\cdots+2^{n-1}=0$ in $R$, the set $\Lambda=\{2^j \colon 0\leq j<n-1\}$ is dissociated: a largest power exceeds the sum of all preceding powers, and every signed sum has absolute value less than $N$. 
The same holds for $t\Lambda$ when $t$ is a unit. 
Applying Chang's inequality~\cite[Lemma~3.1 and its proof]{Chang2002} to $t\Lambda$ gives
    \[
         (n-1)\abs{\widehat{S_w}(t)}^2 \ll \delta_w^2\log(1/\delta_w),
         \qquad
         \abs{\widehat{S_w}(t)}
           \ll \delta_w\sqrt{\frac{\log(1/\delta_w)}{n-1}}.
    \]
Thus, dissociation allows Chang's inequality to replace Parseval. 
In the central window, the two pointwise bounds are respectively $\mybigO C{n^{-3/4}}$ and $\mybigO C{n^{-1}\sqrt{\log n}}$. 
If $N$ is prime, a Mersenne prime, they apply to every nonzero frequency. 
For composite $N$, a nonunit frequency can have a shorter orbit, and multiplication by it need not preserve
dissociation. 
Our fourth-moment estimate has no such restriction.

The need for a fourth moment is not merely an artifact of these particular pointwise estimates. 
Put $M_w=\max\limits_{t\neq 0}\abs{\widehat {S_w}(t)}$. 
For three unit dilates of one slice, bounding one Fourier factor by $M_w$ and applying Parseval to the other two gives
    \[
         \abs{ r_{a_1S_w+a_2S_w+a_3S_w}(x)-N^2\delta_w^3 \vphantom{2^{2^2}} }
               \leq N^2\cdot M_w\delta_w.
    \]
For this upper bound to be $\smallO{N^2\delta_w^3}$, one needs $M_w=\smallO{\delta_w^2}$, or $M_w=\smallO{n^{-1}}$ in the central window. 
However, the proof of \theoremref{lower} gives $\abs{\widehat{S_{w_n}}(-1)} \gg n^{-1}$. 
Thus even the best bound on the largest coefficient, combined only with Parseval in this way, cannot yield the ternary asymptotic for this family. 
The fourth moment retains information about the distribution of Fourier mass that this reduction discards.

A related distinction appears for multiplicative subgroups. 
Let $H\leq\FF_p^\times$ have order $h$, where $p$ is prime. 
The classical square-root character-sum estimate~\cite[Lemma~2.7]{AlonBourgain2014}, followed by Fourier inversion and Parseval, gives
    \[
     \begin{aligned}
         r_{a_1H+a_2H+a_3H}(x) &=\frac{h^3}{p}+\bigO{ p^{1/2}h},
               &&a_1,a_2,a_3\in\FF_p^\times,\\
         \abs{\FF_p\setminus\supp{H+H} }&\ll\frac{p^3}{h^3}.
     \end{aligned}
    \]
At the scale $h\asymp p^{3/4}$, the binary estimate leaves only $\bigO h=\smallO p$ exceptional residues, whereas the ternary error bound is of the same order as its main term. 
This comparison concerns the information supplied by the estimates, not comparable density thresholds for subgroups and slices; central Hamming slices are substantially denser.

\subsection{Cyclic Carries and the Proof Overview}
\label{sec:proof-roadmap}

The additive-combinatorial stage deduces \theoremref{uniform-rep} and \theoremref{full-supp} from the fourth-moment estimate.
For the ternary count, we interpolate between the second and fourth Fourier moments and apply H\"older's inequality.
For binary support, we first apply Cauchy--Schwarz to bound the mixed fourth moment. 
We then apply it to the representation function.
Unit dilations only permute the frequencies. 
The main task is therefore to estimate the energy of one undilated slice. 
Our method is structurally different from the analytic number-theoretic approaches
in~\cite{MauduitSarkozy1997,Konyagin2001,MauduitPomeranceSarkozy2005,MauduitRivat2010,MauduitRivatSarkozy2017,Maynard2019,Green2025}: we count collisions through their cyclic carries, rather than estimate individual Fourier coefficients.

\paragraph{From collisions to a transfer matrix.}
At one digit position, binary addition satisfies $a+b+c=o+2c'$, where $a,b$ are input bits, $o$ is the output bit, and $c,c'$ are the incoming and outgoing carry bits. 
For a collision between two additions, we require a common output and record both carries. 
The resulting automaton has four states. 
Its transition matrix $B(\bu)$, with $\bu=(u_1,u_2,u_3,u_4)$, assigns the monomial $u_1^a u_2^b u_3^g u_4^h$ to input bits $a,b,g,h$, so the four variables record the four input weights.

The end-around carry becomes a cyclic boundary condition expressed by a trace, and \propref{trace-upper-bound-any-n} gives, for every $n\geq 3$,
\[
    E(S_w)\leq [u_1^w u_2^w u_3^w u_4^w]\trace\left(B(\bu)^n\right),
\]
with equality when $w\neq n/2$.
Here brackets denote coefficient extraction. 
The only possible double count occurs when both integer sums equal $N=2^n-1$.
Each input pair then consists of complementary strings, so equal weights would require $2w=n$. 
This possible overcount is harmless: the proof uses only the upper bound above, together with
$E(S_w)\geq \abs{S_w}^4/N$.

\paragraph{\textbf{The cancellation behind the energy estimate.}}
At $u=\mathbf1=(1,1,1,1)$, the eigenvalues of $B(\bu)$ are $8,4,4,2$. 
Let $\Lambda(\bu)$ be the analytic eigenvalue issuing from $8$, and compare it with
    \[
         \Lambda_0(\bu)=\frac12\prod_{j=1}^4(1+u_j),
         \qquad
         [u_1^w u_2^w u_3^w u_4^w]\Lambda_0(\bu)^n
               =2^{-n}\binom nw^4.
    \]
The key fact is that, near $\mathbf1$, the two functions agree whenever any one variable is $1$. Consequently,
    \[
     \Lambda(\bu)-\Lambda_0(\bu)=\prod_{j=1}^4(1-u_j)\cdot H(\bu),
    \]
where $H$ is analytic near $\mathbf1$. 
On the coefficient-integration torus $u_j=\exp({\imath\theta_j})$, this is a fourth-order cancellation. 
Together with the Gaussian decay of the two leading terms, it bounds the local contribution by
    \[
     N^3n\int_{[-\pi,\pi]^4} \norm{\btheta}_2^4
                    \exp\left({-cn\norm{\btheta}_2^2}\right) \, \dd\btheta
           \asymp N^3n^{-3}.
    \]
The Gaussian factor localizes the integral to the scale $\norm{\btheta}_2\asymp n^{-1/2}$.
The four-dimensional Gaussian integral contributes a factor $n^{-2}$, giving the central energy scale $N^3n^{-2}$.
The transfer matrix is iterated once for each of the $n$ binary digits, so the relevant comparison is between $\Lambda^n$ and $\Lambda_0^n$.
Thus the fourth-order cancellation supplies the factor $n\norm{\btheta}_2^4$, which is of order $n^{-1}$ on the Gaussian scale.
This gives the error bound $\bigO{N^3n^{-3}}$.

Away from the origin, the Frobenius--Wielandt comparison theorem \cite[Theorem 2.14, p. 31]{berman1994nonnegative}, together with an analysis of its equality case for $B(\bu)$, gives a uniform spectral gap. 
This part of the integral is exponentially smaller. 
Replacing $2^{-n}$ by $N^{-1}$ in the baseline is absorbed by the stated error. 

Finally, the coefficient formula involves $w$ only through the phase factor $\exp(-\imath w(\theta_1+\cdots+\theta_4))$. 
Since this factor has absolute value one, the triangle inequality bounds the coefficient by an integral that no longer depends on $w$. 
This gives the uniformity in $w$ asserted in \theoremref{energy}.

\paragraph{\textbf{Why the error is attained.}}
For the lower bound, the same cyclic symmetry has the opposite role: it repeats a large coefficient rather than constraining one. 
The $n$ frequencies $-2^j$, $0\leq j<n$, are distinct and have the same coefficient.
For $w_n=\floor{ n/2+\sqrt n }$, we show that $\abs{\widehat {S_{w_n}}(-1)} \gg n^{-1}$, so this single orbit contributes $\gg n^{-3}$ to the fourth moment. 
To estimate the coefficient, we keep only the first $\bigO{\log n}$ most significant digits of a uniform element of $S_{w_n}$ and approximate them by independent bits with bias $w_n/n$.
The resulting Fourier product has magnitude $\gg n^{-1/2}$; the approximation error is smaller. 
Multiplication by the slice density, also of order $n^{-1/2}$, gives the required coefficient.

\paragraph{\textbf{Relation to other methods.}}
The ambient addition law is essential. 
In $\FF_2^n$, Hamming spheres are governed by the Hamming association scheme and Krawtchouk polynomials \cite{Delsarte1973,MacWilliamsSloane1977,Polyanskiy2019}. 
Whereas harmonic analysis on a fixed slice is naturally organized by the Johnson association scheme~\cite{Filmus2016}.  
Related fourth-moment questions are studied in~\cite{Aaronson2018,KirshnerSamorodnitsky2020}. 
Under coordinatewise addition in $\ZZ^n$, energy inequalities for subsets of discrete cubes are developed in~\cite{KaneTao2017,DiosPontGreenfeldIvanisviliMadrid2023,Kovac2023,Shao2026}. 
Neither setting includes cyclic carries.

Automata and transducers are standard tools for digit-defined objects \cite{AlloucheShallit2003,HeubergerKropfProdinger2015}, and carry matrices have a classical combinatorial and probabilistic theory~\cite{Holte1997,DiaconisFulman2009}. 
Cusick also obtains a cyclic carry transfer identity in a related digit problem
\cite[Theorem~6.2, equation~(35)]{Cusick2026}. 
Here, the decisive features are the four-input collision matrix and the fourth-order agreement of its leading eigenvalue with the independent-digit baseline. 
The optimal excess-energy scale comes from this cancellation, not from the transfer-matrix description alone.

\paragraph{\textbf{Organization.}}
\sectionref{prelims} records the Fourier conventions. 
\sectionref{uniform-rep} and \sectionref{full-supp} deduce \theoremref{uniform-rep} and \theoremref{full-supp} from \theoremref{energy}. 
\sectionref{gen} constructs the carry automaton, and \sectionref{energy} proves the energy estimate using the local cancellation and the spectral gap. 
\sectionref{lower} proves \theoremref{lower}.

%%% AI disclosure
\subsection{AI Disclosure: Use of AI Tools}
\label{sec:ai}

We used \textsc{Google AI Ultra} (Gemini) and \textsc{ChatGPT Pro} in a limited, auxiliary role.
%%%
\begin{enumerate} 
    \item Gemini suggested the carry-automaton viewpoint for \theoremref{energy}.
    \item ChatGPT had suggested using Frobenius--Wielandt comparison theorem \cite[Theorem 2.14, p. 31]{berman1994nonnegative} to bound the integral away from the origin. 
    \item ChatGPT was used for English-language editing, including improvements to
        grammar, clarity, organization, and presentation.
\end{enumerate}
%%%
All AI-generated suggestions were independently examined, revised, and incorporated at the authors' discretion.
The authors verified the mathematical arguments and take full responsibility for the correctness and final presentation of the paper.

\section{Notations and Conventions}
\label{sec:prelims}

Throughout the paper,
\[
    N=2^n-1,
    \qquad
    R=\ZZ/N\ZZ,
    \qquad
    W=\{0,1,\ldots,n-1\},
\]
where \(n\geq 3\). 
All sums and congruences involving elements of \(R\) are understood modulo \(N\). 
When a frequency \(t\in R\) occurs in a sum, the notation \(t\neq 0\) means \(t\in R\setminus\{0\}\).

For \(x\in R\), put
\[
    \mye_N(x)\defeq\exp\left(\frac{2\pi \imath x}{N}\right).
\]
For a function \(f\colon R\to\mathbb C\), we use the normalized Fourier transform
\[
    \widehat f(t)
    \defeq
    \frac1N\sum_{x\in R}f(x)e_N(-xt).
\]
Thus, Fourier inversion and Parseval's identity take the forms
\begin{align}
    f(x)
    & =
    \sum_{t\in R}\widehat f(t) \mye_N(xt),
    \label{eq:fourier-inversion} \\
    \sum_{t\in R}\widehat f(t)\overline{\widehat g(t)}
    & =
    \frac1N\sum_{x\in R}f(x)\overline{g(x)}.
    \label{eq:parseval}
\end{align}
For a set \(A\subseteq R\), we write $A(x) = 1_A(x)$ to denote the indicator function, and
\[
    \widehat A\defeq\widehat{1_A},
    \qquad
    \delta(A) \defeq\frac{\abs A}N.
\]
In particular,
\[
    \widehat A(0)=\delta(A).
\]
For the Hamming slice \(S_w\), we abbreviate
\[
    \delta_w \defeq \delta(S_w) = \frac{\abs {S_w}}N.
\]

We use the unnormalized convolution
\[
    (f*g)(x) \defeq \sum_{y\in R}f(y)g(x-y).
\]
With the above Fourier normalization,
\[
    \widehat{f*g}(t)
    =
    N\widehat f(t)\widehat g(t).
\]
Consequently, for \(A_1,\ldots,A_k\subseteq R\),
\begin{equation}
    r_{A_1+\cdots+A_k}(x)
    =
    N^{k-1}
    \sum_{t\in R}
        \widehat{A_1}(t)\cdots\widehat{A_k}(t) \cdot \mye_N(xt).
    \label{eq:representation-fourier}
\end{equation}
In particular,
\begin{equation}
    E(A)
    =
    N^3\sum_{t\in R}\abs{\widehat A(t)}^4
    =
    \frac{\abs A^4}{N}
      +N^3\sum_{t\neq0}\abs{\widehat A(t)}^4.
    \label{eq:energy-fourier}
\end{equation}

For \(a\in R^\times\), write
\[
    aA \defeq \{ax \colon x\in A\}.
\]
A change of variables gives
\begin{equation}
    \widehat{\left(aA\right)}(t)=\widehat A(at).
    \label{eq:dilation-fourier}
\end{equation}
Since \(t\mapsto at\) permutes \(R\), every Fourier moment of \(\widehat A\) is invariant under multiplication of \(A\) by a unit.

The notation \(X\ll Y\) means that \( X\leq CY\) for an absolute constant \(C>0\). 
A subscript, as in \(X\ll_C Y\), records the allowed dependence of the implicit constant on $C$. 
All occurrences of \(o(1)\) refer to the limit \(n\to\infty\).
\section{Near-Uniform Representation Counts}
\label{sec:uniform-rep} 

This section proves \theoremref{uniform-rep} using \theoremref{energy}. 
Note that, by \equationref{representation-fourier} and \ref{eq:dilation-fourier},
\[
    r_{a_1\cdot S_{w_1} + a_2\cdot S_{w_2} + a_3\cdot S_{w_3} }(x) 
        = \sum_{0\leq t< N} N^2 \cdot \widehat{S_{w_1}}(a_1t) \widehat{S_{w_2}}(a_2t) \widehat{S_{w_3}}(a_3t) \cdot \mye_N(xt)
\]
We first separate the $t = 0$ term
%%%
\begin{align*}
    r_{a_1\cdot S_{w_1} + a_2\cdot S_{w_2} + a_3\cdot S_{w_3} }(x)
    & = \frac{\abs{S_{w_1}}\abs{S_{w_2}}\abs{S_{w_3}} }N 
        + \sum_{0< t< N} N^2 \cdot \widehat{S_{w_1}}(a_1t) \widehat{S_{w_2}}(a_2t) \widehat{S_{w_3}}(a_3t) \cdot \mye_N(xt)
    \\
    & = N^2 \cdot \delta_{w_1} \delta_{w_2} \delta_{w_3}
        + \sum_{0< t< N} N^2 \cdot \widehat{S_{w_1}}(a_1t) \widehat{S_{w_2}}(a_2t) \widehat{S_{w_3}}(a_3t) \cdot \mye_N(xt)
\end{align*}
%%%
Upon rearranging, taking absolute values, and applying the triangle inequality, we obtain
%%%
\begin{align*}
    & \hspace{-0.2cm} \abs{ r_{a_1\cdot S_{w_1} + a_2\cdot S_{w_2} + a_3\cdot S_{w_3} } (x) - N^2\cdot \delta_{w_1} \delta_{w_2} \delta_{w_3} } \\
    %%%
        & \hspace{-0.15cm} \leq N^2 \prod_{ j\in\{1,2,3\} } \left( \sum_{t\neq 0} \abs{\widehat{S_{w_j}}(a_jt)}^3 \right)^{1/3}
            \tag{by H\"older's inequality}\\
    %%%
        & \hspace{-0.15cm} \leq N^2 \prod_{j\in\{1,2,3\} } \left( \sum_{t\neq 0} \abs{\widehat{S_{w_j}}(a_jt)}^2 \right)^{1/6}
                \left( \sum_{t\neq 0} \abs{\widehat{S_{w_j}}(a_jt)}^4 \right)^{1/6}
                \tag{interpolating norms $\ell_3\leq \ell_2^{1/3}\cdot \ell_4^{2/3}$}\\
    %%%
        & \hspace{-0.15cm} = N^2 \prod_{j\in\{1,2,3\} } \left( \sum_{t\neq 0} \abs{\widehat{S_{w_j}}(t)}^2 \right)^{1/6}
                \left( \sum_{t\neq 0} \abs{\widehat{S_{w_j}}(t)}^4 \right)^{1/6}
                \tag{because $a_1,a_2,a_3\in R^\times$}\\
    %%%
        & \hspace{-0.15cm} \ll N^2 \cdot \left(\delta_{w_1} \delta_{w_2} \delta_{w_3} \right)^{1/6}\cdot n^{-3/2}
            \tag{by Parseval's identity and \theoremref{energy}}
\end{align*}
%%%
From this, \theoremref{uniform-rep} follows.

\section{Near-Full Support}
\label{sec:full-supp}
In this section, we prove \theoremref{full-supp} using \theoremref{energy}. 
Define
\[
    D \defeq \sum_{t\neq 0} \abs{\widehat{S_{w_1}}(a_1t)}^2 \cdot \abs{\widehat{S_{w_2}}(a_2t)}^2. 
\]
which we can upper bound as follows
%%%
\begin{align}
    D &\leq \prod_{j\in\{1,2\}} \left(\sum_{0<t<N} \abs{\widehat{S_{w_j}}(a_jt)}^4\right)^{1/2}
    \nonumber 
    & (\text{by Cauchy-Schwarz}) \\
    %%%
    &= \prod_{j\in\{1,2\}} \left(\sum_{0<t<N} \abs{\widehat{S_{w_j}}(t)}^4\right)^{1/2}
        & (\text{because $a_1,a_2\in R^\times$})
    \nonumber \\
    %%%
    &\ll n^{-3}
        & (\text{by \theoremref{energy}})
    \label{eq:D-upper-bd}
\end{align}
%%%

For brevity, write $r(x) := r_{a_1\cdot S_{w_1} + a_2\cdot S_{w_2} }(x)$ and $S := \supp{a_1\cdot S_{w_1} + a_2\cdot S_{w_2} }$.
Then using \equationref{representation-fourier}, we can rewrite $r(x)$ in Fourier-analytic form:
\[
    r(x) = N \sum_{0\leq t < N} \widehat{S_{w_1}}(a_1t) \widehat{S_{w_2}}(a_2t)\cdot \mye_N(xt) 
\]
Then, by Parseval's identity 
%%%
\begin{align*}
    \sum_{x\in R} r(x)^2 &= N^3 \sum_{0\leq t<N} \abs{\widehat{S_{w_1}}(a_1t)}^2 \cdot \abs{\widehat{S_{w_2}}(a_2t)}^2
    %%%
    = N^3 \cdot \left( \delta_{w_1}^2 \delta_{w_2}^2 + D \right)
\end{align*}
%%%
Applying Cauchy-Schwarz with the expression above gives
%%%
\begin{align}
    \left( \sum_{x\in R} r(x) \right)^2 &\leq \abs{S} \cdot \sum_{x\in R}r(x)^2 
    = \abs{S} \cdot N^3 \cdot \left( \delta_{w_1}^2 \delta_{w_2}^2 + D\right)
    \label{eq:rx-upper-bound-cauchy}
\end{align}
%%%
On the other hand, we have
\begin{align}
    \sum_{x\in R} r(x) = N^2\cdot \delta_{w_1} \delta_{w_2}. 
    \label{eq:rx-sum}
\end{align}
Combining these two (\equationref{rx-upper-bound-cauchy} and \ref{eq:rx-sum}), we get:
%%%
\begin{equation}
    \abs S \geq N\cdot \frac{\delta_{w_1}^2 \delta_{w_2}^2}{\delta_{w_1}^2 \delta_{w_2}^2 + D}.
    \label{eq:lower-bound-cardinality}
\end{equation}
%%%
Therefore, using the upper bound $D \ll n^{-3}$ (\equationref{D-upper-bd}), the cardinality of the complement satisfies
%%%
\begin{align*}
    \abs{R\setminus S} 
    & \leq N\cdot \frac D{\delta_{w_1}^2 \delta_{w_2}^2 + D}
    \nonumber 
    \tag{by \equationref{lower-bound-cardinality}}
    \\
    %%%
        & \leq N\cdot \min\left\{ 1, \frac D{\delta_{w_1}^2 \delta_{w_2}^2}\right\} 
            \tag{because $\frac v{u+v}\leq \min\left\{1 , \frac vu \right\} $}\\
    %%%
        & \ll N\cdot\min\left\{ 1 , \frac{n^{-3}}{\delta_{w_1}^2 \delta_{w_2}^2}\right\}. 
            \tag{by \equationref{D-upper-bd}}
\end{align*}
%%%
This completes the proof of \theoremref{full-supp}.

\section{Generating Function for Carry Automata}
\label{sec:gen}

Let us construct the generating function for the finite automata for carry addition (refer to \figureref{adder-automaton} for the presentation below). 
Let $a,b,c,c',o\in \bin$ and $u_1$ and $u_2$ be indeterminates. 
At the current position, the incoming carry $c$ and the input bits $a$ and $b$ produce the output bit $o$ and the outgoing carry $c'$, so that over $\ZZ$, 
$$
a+b+c = o + 2c'.
$$

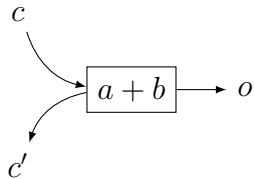
\begin{figure}[H]
    \centering
    \begin{tikzpicture}[scale = 1]
    \node (c) at (0,1) {$c$};
    \node (c') at (0,-1) {$c'$};
    \node [draw] (ab) at (1.5,0) {$a+b$};

    \draw[-latex] (c) [right] to[bend right] (ab) [left];
    \draw[-latex] (ab) [left] to[bend right] (c') [right];

    \node (0) at (3, 0) {$o$};
    \draw[-latex] (ab) [right] to (0)[left];
    \end{tikzpicture}
    \caption{Bitwise addition as an automaton. Incoming carry $c$ and input bits $a$ and $b$ produce output bit $o$ and outgoing carry $c'$, subject to $a+b+c=o+2c'$.}
    \label{fig:adder-automaton}
\end{figure}

Towards constructing the generating function, we construct a $2\times 2$ matrix $M\p o$ as follows:
%%%
\[
    M\p o (u_1,u_2)[c,c'] \defeq \sum_{a,b\colon a+b+c=o+2c'} u_1^a u_2^b.
\]
So, for example, we have:
\begin{align*}
    M\p0(u_1,u_2) = \left(\begin{matrix}
        1 & u_1u_2 \\
        0 & u_1 + u_2
    \end{matrix}\right), \qquad\text{and}\qquad
    M\p1(u_1,u_2) = \left(\begin{matrix}
        u_1 + u_2 & 0 \\
        1 & u_1u_2
    \end{matrix}\right).
\end{align*}
%%%

%%%
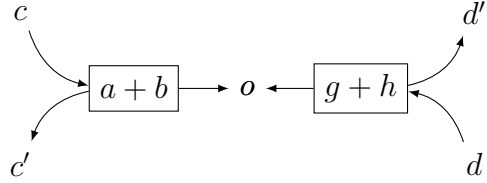
\begin{figure}[H]
    \centering
    \begin{tikzpicture}
    \begin{scope}[shift={(-6,0)}]
        \node (c) at (0,1) {$c$};
        \node (c') at (0,-1) {$c'$};
        \node [draw] (ab) at (1.5,0) {$a+b$};
    
        \draw[-latex] (c) [right] to[bend right] (ab) [left];
        \draw[-latex] (ab) [left] to[bend right] (c') [right];
    
        \node (0) at (3, 0) {$o$};
        \draw[-latex] (ab) [right] to (0)[left];
    \end{scope}
    
    \begin{scope}[scale = -1]
        \node (d) at (0,1) {$d$};
        \node (d') at (0,-1) {$d'$};
        \node [draw] (gh) at (1.5,0) {$g+h$};
    
        \draw[-latex] (d) [right] to[bend right] (gh) [left];
        \draw[-latex] (gh) [left] to[bend right] (d') [right];
    
        \node (0) at (3, 0) {$o$};
        \draw[-latex] (gh) [right] to (0)[left];
    \end{scope}
    \end{tikzpicture}
    \caption{Composition of two automata: the collision-in-addition automaton. The additions $a+b$ with incoming carry $c$ and $g+h$ with incoming carry $d$ produce the same output bit $o$, with outgoing carries $c'$ and $d'$ respectively, subject to $a+b+c = o + 2c'$ and $g+h+d = o + 2d'$.}
    \label{fig:collision-automaton}
\end{figure}

Next, our objective is to construct the generating function of the following ``collision-in-addition'' automata (see \figureref{collision-automaton}): 
Suppose we have two incoming carry bits $c,d\in\bin$. 
Suppose there are two pairs of bits $(a,b),(g,h)\in\bin^2$. 
The addition of $a,b,c$ produces output $o$ and outgoing carry bit $c'$. 
The addition of $ g,h,d$ produces the {\em same} output $o$ and outgoing carry bit $d'$. 
We want to construct a $4\times 4$ generating function matrix $B$ in the indeterminates $\bu=(u_1, u_2, u_3, u_4)$ with the following semantics:
\[
    B(\bu)[cd, c'd'] \defeq \sum_{
            \begin{smallmatrix}
                a,b,g,h,o\in\bin\\
                a+b+c = o + 2c'\\
                g+h+d = o + 2d'
            \end{smallmatrix}
    } u_1^au_2^bu_3^gu_4^h.
\]
%%%
Here $cd\in\{00, 01, 10, 11\}$ and $c'd'\in\{00, 01, 10, 11\}$.

\begin{proposition}
\label{prop:matrix-B-def}
    \[
        B(\bu) = M\p0(u_1,u_2) \otimes M\p0(u_3,u_4) + M\p1(u_1,u_2) \otimes M\p1(u_3,u_4). 
    \]
\end{proposition}
%%%
Specifically,
\begin{align*}
    B(\bu) &= \left(\begin{matrix}
        1 & u_3u_4 & u_1u_2 & u_1u_2\cdot u_3u_4 \\
        0 & u_3 + u_4 & 0 & u_1u_2\cdot (u_3+u_4)\\
        0 & 0 & (u_1+u_2) & (u_1+u_2)\cdot u_3u_4 \\
        0 & 0 & 0 & (u_1+u_2)\cdot (u_3+u_4)
    \end{matrix}\right)
    \\
    &\qquad + \left(\begin{matrix}
        (u_1+u_2)\cdot(u_3+u_4) & 0 & 0 & 0 \\
        (u_1+u_2) & (u_1+u_2)\cdot u_3u_4 & 0 & 0\\
        (u_3+u_4) & 0 & u_1u_2\cdot(u_3+u_4) & 0 \\
        1 & u_3u_4 & u_1u_2 & u_1u_2\cdot u_3u_4
    \end{matrix}\right).
\end{align*}
%%%

Note that $B(\bu)^n$ denotes the generating function of the additions of $n$-bit strings with the semantics that 
$$
B(\bu)^n[cd,c'd'] = \sum\limits_{\bw\in \{0,1,\dots,n\}^4} \nu(\bw)\cdot u_1^{\bw_1}u_2^{\bw_2}u_3^{\bw_3}u_4^{\bw_4}
$$ 
where $\nu(\bw)$ denotes the {\em number} of $n$-bit strings $\ba,\mathbf b,\mathbf g,\mathbf h\in \bin^n$ such that 
%%%
\begin{enumerate}
    \item Their respective Hamming weights are $\bw_1$, $\bw_2$, $\bw_3$, and $\bw_4$.
    \item Starting with carry bit $c$, the addition of $\ba$ and $\mathbf b$ produces some output $\mathbf o\in\bin^n$ and outgoing carry bit $c'$. 
    \item Starting with carry bit $d$, the addition of $\mathbf g$ and $\mathbf h$ produces the {\em same} output $\mathbf o$ and outgoing carry bit $d'$.
\end{enumerate} 
%%%

Fix arbitrary weights $\bw\in W^4$. 
For brevity, $\bu^\bw$ will denote $u_1^{\bw_1}u_2^{\bw_2}u_3^{\bw_3}u_4^{\bw_4}$. 
Our objective is to count the cardinality of the following set:
\[
    \left\{ \vphantom{2^{2^2}}\; (\ba,\mathbf b,\mathbf g,\mathbf h)\in S_{\bw_1}\times S_{\bw_2}\times S_{\bw_3}\times S_{\bw_4}  \;\colon\; \ba + \mathbf b = \mathbf g + \mathbf h\pmod N \; \right\}.
\]
%%%
Since $N=2^n-1$, we consider runs whose starting and outgoing carries
are identical, i.e., $cd=c'd'$. 
So, we are interested in the quantity
\[
    [\bu^\bw] \left( \trace(B(\bu)^n) \right).
\]
However, $[\bu^\bw] \left( \trace(B(\bu)^n) \right)$ {\em may not be identical to the cardinality of the set above}; there is a subtle disconnect. 

Note that in $\ZZ/(2^n-1)\ZZ$ there are two representations of $0$, one is $\underbrace{00\dotsi 0}_{n\text{-times}}$ and another is $\underbrace{11\dotsi 1}_{n\text{-times}}$. Every modular collision is counted once by the trace coefficient, except when $\ba + \mathbf{b} = \mathbf{g} + \mathbf{h} = N$ over $\ZZ$. Such a collision is counted twice, with common outputs $11\cdots 1$ and $00\cdots 0$ corresponding to the incoming carry pairs $(0,0)$ and $(1,1)$, respectively.

%%%
Keeping our downstream applications in mind, consider the special case $\bw = (w,w,w,w)$ for some $w\in W$, and write $\bu^w$ for  $\bu^{(w,w,w,w)}$. If $\ba, \mathbf{b} \in S_w$ satisfy $\ba + \mathbf{b} = N$, then they are bitwise complements, so $2w=n$. Thus the exceptional case cannot occur when $n$ is odd. This gives the following proposition; the full counting argument is provided in \sectionref{trace-upper-bound-for-any-n}.

\begin{proposition}
\label{prop:trace-upper-bound-any-n}
For any $n$, the {\em energy} of the set $S_w$, denoted by $E(S_w)$, satisfies  $E(S_w)\leq [\bu^{\mathbf{w}}]\left(\trace\left( B(\bu)^n\right)\right)$. The equality holds for odd $n$.
\end{proposition}

Recall that our objective is to prove that 
\[
    E(S_w) = \frac1N\cdot \binom nw^4 + \bigO{N^3\cdot n^{-3}}.
\]
Also, note that 
\begin{equation}\label{eq:baseline}
    [\bu^w] \bigg(\!
    \underbrace{\frac12 \cdot (1+\bu_1)(1+\bu_2)(1+\bu_3)(1+\bu_4)}_{\eqdef \Lambda_0(\bu)}
    \!\bigg)^n = \frac1{2^n} \cdot\binom nw^4. 
\end{equation}
%%%
Interpret $\Lambda_0(\bu)$ as the ``ideal baseline polynomial''.
%%%
Let
%%%
$
L_N \defeq \frac1N\binom{n}{w}^4$  
and $L_{2^n}:=2^{-n}\binom{n}{w}^4.$
%%%
From the Fourier-analytic expansion of energy (\equationref{energy-fourier}), we know, 
%%%
\begin{align*}
E(S_w)= N^3\cdot \sum_{t} \left|\widehat{S}_w(t)\right|^4 &= \frac1N\binom{n}{w}^4 + N^3\cdot \sum_{t \neq 0} \left|\widehat{S}_w(t)\right|^4 \geq L_N > L_{2^n}.
\end{align*}
%%%
Therefore, by \propref{trace-upper-bound-any-n}, we obtain
\begin{align*}
& 0\leq E(S_w) - L_N \leq E(S_w) - L_{2^n} \leq [\bu^w] \left( \trace(B(\bu)^n) \right) - L_{2^n}
\end{align*}
Hence, it suffices to prove
$$
    \left| [\bu^w](\trace(B(\bu)^n)) - L_{2^n} \right| \ll N^3 \cdot n^{-3},
$$
which is the content of \sectionref{energy}.

\medskip
In other words, for any $w\in W$, our goal is now to prove:
\[
     \abs{\vphantom{2^{2^2}}\;  [\bu^w] \left( \trace( B(\bu)^n) - \Lambda_0(\bu)^n \right) \; } \ll N^3\cdot n^{-3}.
\]
%%%
We will proceed using Cauchy's coefficient estimation technique by integrating along the torus $\bu=\left(\exp(\imath\theta_1),\exp(\imath\theta_2),\exp(\imath\theta_3), \exp(\imath\theta_4)\right)$, where $\btheta = (\theta_1,\theta_2,\theta_3,\theta_4)\in [-\pi,\pi]^4$. 
Of specific interest would be the behavior of our polynomial in a small neighborhood around $\btheta=\mathbf 0 = (0,0,0,0)$. 
Keeping this end objective in mind, for brevity, we introduce a slight abuse of notation: $\trace(B(\btheta)^n) - \Lambda_0(\btheta)^n$ will denote the polynomial with the substitution $\bu_i=\exp(\imath\btheta_i)$, for every $i\in\{1,2,3,4\}$.  
\subsection{Trace as an Upper Bound of Energy}
\label{sec:trace-upper-bound-for-any-n}

This section proves \propref{trace-upper-bound-any-n}. 

Write $\bu^w$ for $\bu^{(w,w,w,w)}$. 
Recall that $B(\bu)^n[cd,c'd']$ enumerates the tuples $(\ba,\mathbf{b},\mathbf{g},\mathbf{h},\mathbf{o}) \in (\{0,1\}^n)^5$, together with the sequence of carries $c_0,\ldots,c_n$ and $d_0,\ldots,d_n$ in $\{0,1\}$, satisfying 
\begin{equation}
\label{eq:local}
a_i+b_i+c_i=o_i+2c_{i+1}, \quad g_i+h_i+d_i=o_i+2d_{i+1} \qquad (\text{for } 0 \leq i\leq n-1),    
\end{equation}
with $(c_0,d_0)=(c,d)$ and $(c_n,d_n)=(c',d')$, weighted by $u_1^{\wt(\ba)}u_2^{\wt(\mathbf{b})}u_3^{\wt(\mathbf{g})}u_4^{\wt(\mathbf{h})}$. Taking the trace restricts to the runs with $(c_n,d_n)=(c_0,d_0)$, and extracting $[\bu^{w}]$ fixes all four weights to be $w$.

Call a pair $(c,d)\in\{0,1\}^2$ \emph{admissible} for $(\ba,\mathbf{b},\mathbf{g},\mathbf{h})\in S_w^4$ if both additions return to their incoming carries and produce a common output string. 
Since the inputs and incoming carries uniquely determine both additions, each admissible pair determines exactly one run. Writing $P(\ba,\mathbf{b},\mathbf{g},\mathbf{h})$ for the number of admissible pairs, we have
\begin{equation}\label{eq:trace-as-sum}
[\bu^w]\left(\trace\left(B(\bu)^n\right)\right)=\sum_{\ba,\mathbf{b},\mathbf{g},\mathbf{h} \in S_w} P(\ba,\mathbf{b},\mathbf{g},\mathbf{h}).    
\end{equation}

Multiplying the first relation in \equationref{local} by $2^{i}$ and summing over $0\leq i\leq n-1$, the intermediate carry terms cancel, giving 
\begin{equation}
\label{eq:aggregate}
\mathbf{a} + \mathbf{b} + c_0=2^n c_n + \mathbf{o}    
\end{equation}
where the strings are identified with the integers they represent, with position $i$ having weight $2^i$. Similarly, $\mathbf{g} + \mathbf{h} + d_0 = 2^n d_n + \mathbf{o}$.

Fix $\ba,\mathbf{b},\mathbf{g},\mathbf{h} \in S_w$, write 
$$
s_1=\ba + \mathbf{b} \text{ over } \ZZ \quad \text{and} \quad s_2=\mathbf{g} + \mathbf{h} \text{ over } \ZZ
$$ 
Since $w\in W=\{0,\ldots,n-1\}$, we have $0\leq s_1,s_2\leq 2N-2$.

For an addition with sum $s$ and incoming carry $c$, \equationref{aggregate} shows that it returns to $c$ if and only if $0\leq s-Nc \leq N$, with output $s-cN$. In other words,
\begin{equation}
\label{eq:closure}
    \begin{cases}
        c=0 \text{ and } s\leq N, &\text{with output }s, \\
        c=1 \text{ and } s\geq N, &\text{with output }s-N \\
    \end{cases}
\end{equation}
In particular, the common output gives $s_1 - Nc = s_2 - Nd$. 
Hence, $P(\ba,\mathbf{b},\mathbf{g},\mathbf{h})=0$ unless $s_1=s_2\pmod N$.

Now, suppose $s_1=s_2\pmod N$. 
\begin{itemize}[leftmargin=1.7cm]
    \item[\textbf{Case 1.}]
    Suppose $s_1,s_2\neq N$. 
    
    By \equationref{closure}, each addition has a unique carry that returns to itself, and its output lies in $\{0,1,\ldots,N-1\}$. The outputs agree since $s_1=s_2\pmod N$. Hence, $P(\ba,\mathbf{b},\mathbf{g},\mathbf{h})=1$.

    \item[\textbf{Case 2.}]
    Suppose exactly one sum equals $N$, say $s_1=N \neq s_2$. 
    
    The congruence and the bound on $s_2$ force $s_2=0$. This gives $w=0$ and hence $s_1=0$, a contradiction. The other case is symmetric.

    \item [\textbf{Case 3.}]
    Suppose $s_1=s_2= N$. 
    
    By \equationref{closure}, precisely the pairs $(0,0)$ and $(1,1)$ are admissible, with common outputs $N$ and $0$, respectively. Hence, $P(\ba,\mathbf{b},\mathbf{g},\mathbf{h})=2$.
\end{itemize}
Thus every collision is counted once, except when $\ba + \mathbf{b} = \mathbf{g} + \mathbf{h} = N$ over $\ZZ$, in which case it is counted twice. By \equationref{trace-as-sum}, 
$E(S_w)\leq [\bu^{\mathbf{w}}]\left(\trace\left( B(\bu)^n\right)\right)$, with equality if and only if Case~3 does not occur. 
Finally, $\ba+ \mathbf{b} = N$, implies that $\mathbf{b}$ is the bitwise complement of $\ba$, so $2w=\wt(\ba)+\wt(\mathbf{b})=n$. 
This is impossible for odd $n$, proving equality in that case.
\section{Tighter Estimation of Energy}
\label{sec:energy}

In this section, we prove:
%%%
\begin{equation}
\label{eq:target-1}
    \abs{E(S_w) - \binom nw^4\cdot N^{-1}} \ll N^3\cdot n^{-3}, 
\end{equation}
%%%
which is equivalent to \theoremref{energy}'s upper bound
    \[
        \sum_{t\neq 0} \abs{\widehat{S_w}(t)}^4 \ll n^{-3}.
    \]
As alluded in \sectionref{gen} after \propref{trace-upper-bound-any-n}, we make the following innocuous modification:
instead of proving \equationref{target-1} directly,
it would suffice to prove that 
$$
\abs{[\bu^w](\trace(B(\bu)^n)) -\binom nw^4\cdot 2^{-n}} \ll N^3\cdot n^{-3}.
$$
Let us proceed with this estimation now. We set up a generating function
    \[
        G(\bu) := \trace( B(\bu)^n ) - \Lambda_0(\bu)^n, 
    \]
where $\bu=(u_1, u_2, u_3, u_4)$; the matrix $B$ is defined by \propref{matrix-B-def} in \sectionref{gen} (see also \equationref{B-def} for another form) and $\Lambda_0$ is defined by \equationref{baseline}. 
%%%
\begin{claim}
The following inequality holds:
    \[
        0\leq E(S_w) - \binom nw^4\cdot2^{-n} \;\leq\; [u_1^w u_2^w u_3^w u_4^w] \left( G(\bu) \right).
    \]
\end{claim}
%%%
This follows from \propref{trace-upper-bound-any-n} and the discussion after it in \sectionref{gen}.
%%%
By the Cauchy integral formula (or, in this specific case, the orthogonality relation), we have:
\[
    [u_1^w u_2^w u_3^w u_4^w] \left( G(\bu) \right) = \frac1{(2\pi)^4} \int_{[-\pi,\pi]^4} G(\btheta) \exp\left(-\imath w(\btheta_1+\btheta_2+\btheta_3+\btheta_4\right)) \dd\btheta, 
\]
where 
    $G(\btheta) = G\left( \exp(\imath\btheta_1), \exp(\imath\btheta_2), \exp(\imath\btheta_3), \exp(\imath\btheta_4) \right)$. 
By the triangle inequality, we have
\[
    \abs{E(S_w) - \binom nw^4\cdot 2^{-n}} \ll \int_{[-\pi,\pi]^4} \abs{G(\btheta)} \dd\btheta.
\]
We split the contour into two different parts and estimate this integral separately. 
Let us partition the set $[-\pi,\pi]^4$ into two sets.
%%%
\begin{align}
    \bTheta_\close := \left\{ \btheta\in [-\pi,\pi]^4 \;\colon\; \abs{\btheta_1},\abs{\btheta_2},\abs{\btheta_3},\abs{\btheta_4}\leq \tau \right\}
    \;\text{ and }\;
    \bTheta_\far := [-\pi,\pi]^4\setminus \bTheta_\close
    \label{eq:theta-close-far-def}
\end{align}
for some parameter $\tau$. 
%%%

\begin{claim}\label{clm:int-close}
The following estimation holds:
    \[
        \int_{\bTheta_\close} \abs{G(\btheta)} \dd\btheta \ll N^3\cdot n^{-3}.
    \]
\end{claim}

\begin{claim}\label{clm:int-far}
For some positive constant $c = c(\tau)$, 
    \[
        \int_{\bTheta_\far} \abs{G(\btheta)} \dd\btheta \ll N^3\cdot \exp(-cn).
    \]
\end{claim}
%%%
Below, \sectionref{int-close} proves \claimref{int-close} and \sectionref{int-far} proves \claimref{int-far}. 
From these two claims, we immediately conclude \equationref{target-1}.

\subsection{Proof of \claimref{int-close}}
\label{sec:int-close}

We show that there is a function $\Lambda(\btheta)$ such that 
%%%
\begin{align*}
%%%
\abs{G(\btheta)} 
& \leq \abs{\Lambda(\btheta)^n - \Lambda_0(\btheta)^n} + \bigO{5^n}
\\
%%%
& \leq n\cdot 
        \abs{\Lambda(\btheta)-\Lambda_0(\btheta)} \cdot 
        \left(\vphantom{2^{2^2}}\abs{\Lambda(\btheta)}^{n-1} 
                + \abs{\Lambda_0(\btheta)}^{n-1}\right) + 
                        \bigO{5^n}.
%%%
\end{align*}
%%%
We first state the following claim (see \equationref{theta-close-far-def} for the definition of $\bTheta_{\close}$).
%%%
\begin{claim}\label{clm:cent-behavior}
The following inequalities hold:
%%%
    \begin{align*}
    %%%
        \abs{\Lambda(\btheta)-\Lambda_0(\btheta)} & \ll \norm{\btheta}_2^4
            \tag{for $\btheta\in\bTheta_\close$}
        \\
        %%%
        \abs{\Lambda_0(\btheta)} &\leq 8\exp\left(- \frac{\norm{\btheta}_2^2}{\pi^2} \right)
        \\
        %%%
        \abs{\Lambda(\btheta)} &\leq 8\exp\left( - c\norm{\btheta}_2^2\right) 
            \tag{for $\btheta\in\bTheta_\close$}.
    %%%
    \end{align*}
%%%
\end{claim}
\claimref{cent-behavior} gives us
\begin{align*}
    \int_{\bTheta_\close} \abs{G(\btheta)} \dd\btheta 
    %%%
    &\ll N^3\int_{\bTheta_\close}n\cdot \norm{\btheta}_2^4 \cdot \exp \left( -cn\norm{\btheta}_2^2 \right) \dd\btheta +\bigO{5^n} \\
    %%%
    &\ll N^3 \cdot n \cdot n^{-4} + \bigO{5^n} \ll N^3\cdot n^{-3}.
\end{align*}
The proof of \claimref{cent-behavior} is deferred to \sectionref{cent-behavior}.

\subsection{Proof of \claimref{int-far}}
\label{sec:int-far} 

It suffices to show that
\begin{align*}
    & \abs{\trace\left( B(\btheta)^n \right)} \ll N^3\cdot \exp(-cn) \\
    & \abs{\Lambda_0(\btheta)}^n \ll N^3\cdot \exp(-cn)
\end{align*}
These are argued below. 

\subsubsection*{Smallness of $\Lambda_0$.}
Suppose $\abs{\theta_1}\geq \tau$.
Then, 
\[
    \abs{1 + \exp(\imath\theta_1)} \leq 2\cos\left(\frac{\theta_1}{2} \right) \leq 2\cos\left( \frac{\tau}{2}\right) \leq 2\cdot \exp\left(-\frac{\tau^2}{8}\right).
\]
From this, it is immediate that
\[
    \abs{\Lambda_0(\btheta)}^n \ll N^3\cdot \exp\left(-\frac{\tau^2}{8} \cdot n\right).
\]
This argument works for any $\btheta\in\bTheta_\far$.

\subsubsection*{Smallness of $\trace\left( B(\btheta)^n \right)$.}

We will prove the following claim in \sectionref{spec-prop}:
\begin{claim}\label{clm:small-faraway}
    For $\btheta\in\bTheta_\far$, there is a positive constant $c=c(\tau)$ such that the spectral radius is upper-bounded: 
    \[
        \rho(B(\btheta)) \leq 8\cdot\exp(-c).
    \]
\end{claim}
%%%
\claimref{small-faraway} implies that
$$
\trace\left(B(\btheta)^n\right) = \mu_1^n + \mu_2^n + \mu_3^n + \mu_4^n,
$$
where, for all eigenvalues, one has $\abs{\mu_1}, \abs{\mu_2}, \abs{\mu_3}, \abs{\mu_4} \leq 8\exp(-c)$.
We hence conclude 
$$
\abs{\trace\left( B(\btheta)^n \right)} \ll N^3\cdot \exp(-cn)
$$

\subsection{Proof of \claimref{cent-behavior}}
\label{sec:cent-behavior}

% \subsubsection*{Part 1.}
From \equationref{close-behav} in \sectionref{spec-prop}, we have
\[
    \abs{\Lambda(\btheta) - \Lambda_0(\btheta) }
    %%%
    \ll \prod_{j=1}^4\abs{\left(1-\exp(\imath\btheta_j)\right)} 
    %%%
    \leq 16\prod_{j=1}^4 \abs{\sin\left(\frac{\btheta_j}{2} \right)}
    %%%
    \leq \prod_{j=1}^4 \abs{\btheta_j}
    %%%
    \ll \norm{\btheta}_2^4. 
\]

% \subsubsection*{Part 2.}
Note that $\abs{1 + \exp(\imath \theta)}  = 2\cos(\theta/2) \leq 2\exp(-\theta^2/\pi^2)$, because $\cos t\leq \exp(-4t^2/\pi^2)$. 
So, we have 
\[
    \abs{\Lambda_0(\btheta)} \leq 8\cdot \exp\left( - \frac{\norm{\btheta}_2^2}{\pi^2} \right).
\]
and hence
\[
    \abs{\Lambda(\btheta)} 
    %%%
    \leq \abs{\Lambda_0(\btheta)} + \abs{\Lambda(\btheta) - \Lambda_0(\btheta)} 
    %%%
    \leq 8\cdot\exp\left(-\frac{\norm{\btheta}_2^2}{\pi^2} \right) + \bigO{\norm{\btheta}_2^4}.
\]
Fix any constant $c$ satisfying $0<c<1/\pi^2$. Since
$$
8\left(
\exp\left(-c\norm{\btheta}_2^2\right)
-
\exp\left(-\frac{\norm{\btheta}_2^2}{\pi^2}\right)
\right)
\asymp \norm{\btheta}_2^2
\;\text{ as } \btheta\to 0
$$
this difference dominates the
$\bigO{\norm{\btheta}_2^4}$ term for sufficiently small $\norm{\btheta}_2$.  
Therefore, by choosing the threshold $\tau>0$ sufficiently small, we obtain $\abs{\Lambda(\btheta)}\leq 8\cdot\exp\left(-c\norm{\btheta}_2^2\right)$ for all $\btheta\in\bTheta_\close$.

\subsection{Spectral Properties of \texorpdfstring{$B(\btheta)$}{B(theta)} and Proof of \claimref{small-faraway}
}
\label{sec:spec-prop}

In this section, we prove the following spectral properties of the matrix $B(\btheta)$.
\begin{enumerate}
    \item When $\btheta\in\bTheta_\close$, its largest eigenvalue behaves as:
        \begin{equation}\label{eq:close-behav}
            \abs{\Lambda(\btheta) - \Lambda_0(\btheta)} 
                \ll  \abs{\prod_{j=1}^4\left(1-\exp(\imath\btheta_j)\vphantom{2^{2^2}}\right) }.
        \end{equation}
    \item (\claimref{small-faraway}) Its spectral radius is separated from $8$ when $\btheta\in\bTheta_\far$.
        
\end{enumerate}
%%%
In the sequel, we use $\bu=(u_1, u_2, u_3, u_4)$ defined by $u_j:= \exp(\imath \btheta_j)$, for $j\in\{1,2,3,4\}$. 
Recall from \sectionref{gen}, we have:
%%%
\begin{equation}
\label{eq:B-def}
    B(\bu) = \begin{pmatrix}
        pr+1 & s    & q    & qs \\
        p    & ps+r & 0    & qr \\
        r    & 0    & qr+p & ps \\
        1    & s    & q    & qs+pr
    \end{pmatrix}
    \;\text{ where }\;\;
    \begin{aligned}
        p &= u_1+u_2, & q &= u_1u_2, \\
        r &= u_3+u_4, & s &= u_3u_4.
    \end{aligned}
\end{equation}
%%%
Our objective is to investigate the dominant eigenvalue $\Lambda(\bu)$ of this matrix. 

Consider a sufficiently small neighborhood parameterized by $\tau>0$:
\[
    \cN \defeq \left\{ \; \bu \;\colon\; \abs{1 - \bu_j}\leq \tau, \text{ for every }j\in\{1,2,3,4\} \vphantom{2^{2^2}} \; \right\}.
\]
Note that $B(\bm 1)$ has eigenvalues $8, 4, 4, 2$ where $\bm 1 := (1,1,1,1)$.
There is an analytic function $\Lambda(\bu)$ such that its magnitude is $> 7$ and is the largest eigenvalue of $B(\bu)$, for any $\bu\in N$. 
All other eigenvalues of $B(\bu)$ are $< 5$ in magnitude. 

Consider the auxiliary polynomial 
$$
P(\bu;z) := \det( \diag(z,z,z,z) -B(\bu) )
$$
Its roots represent the eigenvalues corresponding to $\bu$. 
Note that $P(\bu;z)$ is a function such that the partial derivative with respect to $z$ at $(\bm 1,8)$ is not equal to $0$. 
Therefore, by the implicit function theorem, there exists an analytic function $\Lambda(\bu)$ such that
\[
    P(\bu; \Lambda(\bu) )=0,\text{ and }\Lambda(\bm 1)=8
\] 
in a small neighborhood of $\bm 1$.

Let us choose two contours:
\[
    \Gamma_{\textrm{high}} \defeq \{z \colon \abs{z-8}=1\} \;\text{ and }\; 
    \Gamma_{\textrm{rest}} \defeq \{z \colon \abs{z} = 5\} .
\]

Now, on both these contours, $\abs{P(\bm 1;z)}=\abs{(z-8)(z-4)^2(z-2)} > 0$. Therefore, for $\bu$ very close to $\bm 1$, since $P(\bu;z)$ changes continuously with $\bm u$, we can make
$$
    \abs{P(\bu;z) - P(\bm 1;z)} < \abs{P(\bm 1;z)}
$$
on each contour. Then, by Rouch\'e's theorem, inside both the contours, $P(\bu;z)$ and $P(\bm 1;z)$ have the same number of roots, counted with multiplicity. Therefore, together with the continuity of spectral radius and the fact that $\Lambda(\bm 1)=8$, we can assume, without loss of generality, by contracting the neighborhood of $\bu$ around $\bm 1$ appropriately (if necessary) that 
\begin{enumerate}
    \item We have an analytic function $\Lambda(\bu)$ denoting the largest eigenvalue of $B(\bu)$ which remains in the proper interior of $\Gamma_{\textrm{high}}$.
    \item All remaining eigenvalues of $B(\bu)$ are in the proper interior of $\Gamma_{\textrm{rest}}$.
\end{enumerate}  

Therefore, we get 
$$
\trace(B(\bu)^n) = \Lambda(\bu)^n + \bigO{5^n}
$$

Recall from \equationref{baseline} that $\Lambda_0(\bu) = \frac12 \prod_{j=1}^4 \left(1+\bu_j\right)$. 
Define $\bu'$ as $\bu$ restricted to $\bu_1=1$. 
Then we have the identity:
\[
    B(\bu')\cdot \bm 1 = \Lambda_0(\bu')\cdot \bm 1.  
\]
Because every row of $B(\bu')$ sums to the same number when $\bu_1 = 1$, and that common number is $\Lambda_0(\bu')$. 
So, $\Lambda_0(\bu')$ is an eigenvalue of $B(\bu')$.
At $\bu'=\bm 1$, we have $\Lambda_0(\bm 1)=8$, which coincides with $\Lambda(\bm 1)$. After shrinking the neighborhood $\cN$ if necessary, $\Lambda_0(\bu')$ lies inside $\Gamma_{\textrm{high}}$. By the
spectral separation above, $B(\bu')$ has exactly one eigenvalue inside
$\Gamma_{\mathrm{high}}$, namely $\Lambda(\bu')$. Therefore, $\Lambda(\bu') = \Lambda_0(\bu')$.

From this, we conclude that $(1-\bu_1)$ divides $\Lambda(\bu) - \Lambda_0(\bu)$. Likewise, by symmetry, each factor
$(1-\bu_j)$ divides $\Lambda(\bu)-\Lambda_0(\bu)$ for $j\in\{1,2,3,4\}$. Consequently, 
\[
    \Lambda(\bu) - \Lambda_0(\bu) = \left(\prod_{j=1}^4 (1-\bu_j)\right) \cdot H(\bu),
\]
for some analytic function $H(\bu)$. 
The maximum magnitude of $H(\bu)$ is bounded in the neighborhood $\cN$, and this proves our first spectral property (\equationref{close-behav}). 

Next, we prove the spectral gap when $\btheta\in\bTheta_\far$. 
For this, we need the following result.
%%%
\begin{theorem}[Frobenius-Wielandt comparison theorem {\cite[Theorem 2.14, p. 31]{berman1994nonnegative}} \!\!]
\label{thm:frobenius-wielandt}
    Let $A$ be an irreducible non-negative matrix, and let $B$ be a complex matrix satisfying $\abs{B_{i,j}} \leq A_{i,j}$, entrywise. 
    Then
    \[
        \rho(B) \leq \rho(A).
    \]
    More generally, every eigenvalue $\gamma$ of $B$ satisfies $\abs \gamma \leq \rho(A)$. 
    If equality holds for some eigenvalues $\gamma$ of $B$, that is, if $\abs\gamma = \rho(A)$, then there exists a diagonal unitary matrix $ D$ and a real number $\phi$ such that
    \[
        B = \exp(\imath\phi) DAD^{-1} \text{ and }\exp(\imath\phi) = \frac{\gamma}{\rho(A)}.
    \]
    In particular, if $\rho(B)=\rho(A)$, then $B = \exp(\imath\phi)DAD^{-1}$ for some real $\phi$ and diagonal unitary matrix $D$. 
\end{theorem}
%%%
We emphasize that this result shows that $\rho(B(\bu))$ is far from $8$ for {\em any} $\bu$ far away from $\bm1$.  

Let $A := B(\bm 1)$.
$B(\bu)$ consists of sums of monomials with coefficients $0$ or $1$, so on the torus we have $|B(\bu)_{i,j}| \leq B(\bm 1)_{i,j} = A_{i,j}$. 
Note that $A$ is non-negative and irreducible because $A^2$ has all non-zero entries. 
Hence, we can invoke \theoremref{frobenius-wielandt} and get
$$
\rho(B(\bu)) \leq \rho(A) = 8.
$$

Let us show that it is impossible to achieve equality for $\bu\neq \bm 1$.
Suppose, for the sake of contradiction that the equality holds 
at some $\bu\neq 1$. 
Then, $B(\bu) = \exp(\imath\phi) DAD^{-1}$ for some real $\phi$.
Since $D$ is unitary and diagonal,
$$
\abs{B(\bu)_{i,j}} = A_{i,j}
$$
Comparing $(2,1)$ elements of $A$ and $B(\bu)$ gets us $\bu_1=\bu_2$, $(3,1)$ elements $\bu_3 = \bu_4$, and $(1,1)$ elements $\bu_3 = \overline{\bu_1}$
Hence, we have $\bu_1=\bu_2 = \overline{\bu_3} = \overline{\bu_4}$, and from this we can conclude that $B(\bu)_{1,4}=1$.

Write $D = \diag(d_1, d_2, d_3, d_4)$. 
Observe that
\begin{align*}
    1= B(\bu)_{1,4} &= \exp(\imath\phi) d_1 \overline{d_4} A_{1,4} = \exp(\imath\phi) d_1 \overline{d_4}\\
    %%%
    1=B(\bu)_{4,1} &= \exp(\imath\phi) d_4 \overline{d_1} A_{4,1} = \exp(\imath\phi) d_4 \overline{d_1}.
\end{align*}
These two constraints imply that $\exp(\imath\phi)\in\{\pm1\}$; i.e., $\exp(\imath \cdot 2\phi)=1$. 
Comparing $(1,2)$ and $(2,1)$ elements of $A$ and $B(\bu)$ gives us $B(\bu)_{1,2} = \overline{\bu_1}^2$ and $B(\bu)_{2,1} = 2\bu_1$ respectively.
We also have 
\begin{align*}
    B(\bu)_{1,2} & = \exp(\imath\phi)d_1 \overline{d_2}A_{1,2}
    \\
    B(\bu)_{2,1} & = \exp(\imath\phi)d_2 \overline{d_1} A_{2,1}
\end{align*}
Upon multiplying them, we get $\overline{\bu_1} = \exp(\imath \cdot 2\phi)$, which we know is $1$.
Altogether, we obtain $\bu_1=1=\bu_2=\overline{\bu_3}=\overline{\bu_4}$, i.e. $\bu=\bm 1$, but this is a contradiction. 

It remains to upgrade this pointwise strict inequality to a uniform separation.
For simplicity, let us revert to the $\btheta$ notation. 
Consider the set 
\[
    S_\tau \defeq \left\{ \btheta\in[-\pi,\pi]^4 \;\colon\; \norm{\btheta}_\infty \geq \tau \right\}.
\] 
Note that $S_\tau$ is a closed subset of the compact set $[-\pi,\pi]^4$. Hence, $S_\tau$ is compact and does not contain
$\mathbf{0}$. Since the spectral radius is a continuous function, the spectral radius of $B(\btheta)$ on $S_\tau$ is strictly smaller than $8$. 
This continuous function attains a maximum at some point in the set $S_\tau$ since it is a compact set. 
Denote this maximum spectral radius in $S_\tau$ by $8\cdot \exp(-c)$, for some positive constant $c$. 
This completes the proof that $\rho(B(\btheta))\leq 8\cdot \exp(-c)$ for $\btheta\in\bTheta_\far$, which is the second desired spectral property (\claimref{small-faraway}). 

\section{Tightness of the Energy Bound}
\label{sec:lower}

Here we prove \theoremref{lower}.
For every $n\geq 5$, we have $w_n\in W$, and $S_{w_n}$ is a nonempty proper subset of $R$. 
Therefore, by Parseval,
\(
  \sum_{t\neq 0}\abs{\widehat{S_{w_n}}(t)}^4>0.
\)
It therefore suffices to prove the lower bound in \theoremref{lower} for all sufficiently large $n$, since the finitely many remaining cases can be absorbed by decreasing the absolute implied constant. 
Throughout the remainder of this section, we assume that $n$ is sufficiently large.

Note that because $N=2^n-1$, we have the following fact: $x\in S_w$ if and only if $2x\in S_w$. 
As a result, we have the identity $\widehat{S_w}(t) = \widehat{S_w}(2t)$.
In particular, we have $\widehat{S_w}(-2^j)$ is identical, for all $j\in\{0,1,\dotsc, n-1\}$. 
For \(0\le j<k<n\), we have
    \[
        0<2^k-2^j<2^n-1=N,
    \]
and hence \(-2^j\not\equiv -2^k\pmod N\). 
Thus these are \(n\) distinct nonzero frequencies.
We remark that in general the set $\{2^j t \;\colon\; 0\leq j<n\}$ need not have cardinality $n$.

After that, restricting to these specific $n$ frequencies, we get:
\[
    \sum_{t\neq 0} \abs{\widehat{S_{w_n}}(t)}^4 \geq n \abs{\widehat{S_{w_n}}(-1)}^4.
\]
Thus, it suffices to prove that 
\[
    \abs{\widehat{S_{w_n}}(-1)} \gg n^{-1}.
\]

To that end, consider $X$ to be uniformly distributed over $S_{w_n}$. 
Note that we can rewrite 
\[
    \widehat{S_{w_n}}(-1) = \frac{\binom n{w_n}}N \cdot \expect{\mye_N(X)}.
\]

After that, we estimate these two parts on the RHS as follows: 
\begin{claim}\label{clm:binom-lower}
    \[
        \binom n{w_n} \gg 2^n\cdot n^{-1/2} \asymp N\cdot n^{-1/2}.
    \]
\end{claim}

\begin{claim}\label{clm:hypergeom-lower}
    \[
        \abs{ \vphantom{2^{2^2}}\expect{\mye_N(X)} } \gg n^{-1/2} .
    \]
\end{claim}

From these two claims, it follows that $\abs{\widehat{S_{w_n}}(-1)} \gg n^{-1}$; hence proving \theoremref{lower}.

\begin{proof}[Proof of \claimref{binom-lower}]
The proof follows from basic facts:
\begin{enumerate}
    \item For the central binomial coefficient, we have~\cite[p. 96, Proposition 3.6.2]{matousek2008invitation}: 
        \[
            \binom {2m}m \asymp 2^{2m} \cdot m^{-1/2}.
        \]
    \item For $0\leq 2i/m \leq 1/2$, the following expression holds:
        \[
            \binom{2m}{m+i} 
                %%%
                = \binom {2m}m \prod_{j=1}^i \frac{m+1-j}{m+j} 
                %%%
                \geq \binom {2m}m \prod_{j=1}^i \left(1 - \frac{2j}{m}\right)
                %%%
                \geq \binom {2m}m  \cdot \exp\left(-4i^2/m\right)
        \]
            Here we have used Jensen's inequality
            \(
                \ln(1-u) \geq -\ln4 \cdot u ,
            \)
            when $u\in[0,1/2]$. 
    \item Finally, by definition:
        \[
            \binom {2m}{m+i} \leq \binom{2m+1}{m+1+i} \leq 2\binom {2m}{m+i}.
        \]
\end{enumerate}
    These three inequalities can be put together to prove the claim. 

    \paragraph{Case: Even $n$.}
    Suppose $n=2m$ and $w_n = \floor{m + \sqrt n}$. 
    Define $i_n = w_n - m = \floor{\sqrt n}$.
    Now, when $2i_n/m \leq 1/2$ (which holds for $n\geq 56$), we have
    \[
        \binom n{w_n} \gg 2^n\cdot n^{-1/2}.
    \]

    \paragraph{Case: Odd $n$.}
    Suppose $n=2m+1$ and $w_n = \floor{m + 1/2 + \sqrt n}$.
    Define $i_n = w_n- (m+1) = \floor{\sqrt n - 1/2}$. 
    Now, when $2i_n/m \leq 1/2$ (which holds for $n\geq 57$), we have
    \[
        \binom n{w_n} \asymp \binom{2m}{m+i_n} \gg 2^n\cdot n^{-1/2}.
    \]
    This completes the proof of \claimref{binom-lower}.
\end{proof}

\begin{proof}[Proof of \claimref{hypergeom-lower}]
    Recall $X$ denotes the uniform distribution over $S_{w_n}$, and consider the binary representation of the integer $X = X_12^{n-1}+\dotsi +X_{n-1}2^1+ X_n2^0$, where $X_1, \dotsc, X_n \in \bin$. 
    The random variable $X_i$ denotes the $i$-th most significant bit in the $n$-bit binary expansion of the integer $X$. 

    Because of the special form of $N=2^n-1$, note that the binary representation of the fraction $X/N\in[0,1)$ has the repeating pattern
    \[
        0.\overline{X_1 X_2 \dotsc X_{n-1} X_n} \;_2.
    \]
    For every integer $0< u \leq  n$, let $\theta_u$ denote the random variable corresponding to the number $0.X_1X_2\dotsc X_u$

    As our first reduction, we claim:
    %%%
    \begin{claim}\label{clm:first-u}
        \[
            \abs{ \expect{\mye_N(X)} - \expect{\exp(2\pi\imath \cdot \theta_u)}  \vphantom{2^{2^2}} } \ll 2^{-u}.
        \]
    \end{claim}

    Thus, it will suffice to estimate $\expect{\exp(2\pi\imath \cdot \theta_u)} $ instead. 
    For the second reduction, we consider the following new random variable: $Y$ is the random variable $Y_12^{-1} + \dotsi + Y_u2^{-u}$, where each $Y_j$ is independently set to $1$ with probability $p_n\defeq w_n/n$; $0$, otherwise. 
    For brevity, we just write this as the random variable $0.Y_1\dotsc Y_u$. 
    Our second reduction claims:
    %%%
    \begin{claim}\label{clm:iid-u}
        \[
            \abs{ \expect{\exp(2\pi\imath \cdot \theta_u)} - \expect{\exp(2\pi\imath \cdot 0.Y_1\dotsi Y_u)} \vphantom{2^{2^2}} } \ll u^2 \cdot n^{-1}.
        \]
    \end{claim}

    Note that the following quantity $P(u)$ is identical to $\expect{\exp(2\pi\imath \cdot 0.Y_1\dotsi Y_u)}$
        \[
            P(u) \defeq \prod_{j=1}^u \left( (1-p_n) + p_n\cdot \exp\left(\frac{2\pi\imath }{2^j}\right)\right)
        \]
    This quantity is estimated below. 
    \begin{claim}\label{clm:ideal-est}
        \[
            \abs{P(u)} \gg \abs{1-2p_n} \gg n^{-1/2}.
        \]
    \end{claim}

    By \claimref{first-u}, \claimref{iid-u}, and \claimref{ideal-est}, there exist absolute constants \(c,C>0\) such that
        \[
            \abs{\expect{\mye_N(X)}\vphantom{2^{2^2}}}
            \geq C n^{-1/2}
                -c\left(2^{-u}+\frac{u^2}{n}\right).
        \]
        Taking \(u=\lceil\log_2 n\rceil\), the second term is \(\smallO{n^{-1/2}}\), and hence
        \[
            \abs{\vphantom{2^{2^2}} \expect{\mye_N(X)}} \gg n^{-1/2}.
        \]
    This gives us \claimref{hypergeom-lower}. 
\end{proof}

%%%
\begin{proof}[Proof of \claimref{first-u}]
    Since \(X/N=0.\overline{X_1\cdots X_n}\;_2\), we have
        \[
            0\leq \frac XN-\theta_u\leq 2^{-u}.
        \]
        Therefore, using
        \(
            \abs{ \exp\left({2\pi \imath x}\right) - \exp\left({2\pi \imath y}\right) }\leq 2\pi\cdot \abs{x-y},
        \)
        \[
        \abs{
                \expect{\mye_N(X)} - \expect{ \exp\left({2\pi \imath\theta_u} \right) } \vphantom{2^{2^2} }
        }
        \leq 2\pi\cdot 2^{-u}.
        \]
\end{proof}

%%% 
\begin{proof}[Proof of \claimref{iid-u}]
    The only difference between the two experiments is that $\theta_u$ corresponds to a ``sampling without replacement'' and $0.Y_1\dotsi Y_u$ corresponds to a ``sampling with replacement'' experiment. 

    Let us formalize it further. 
    Consider $n$ balls, $w_n$ of them are labeled $1$ and the remaining $(n-w_n)$ are labeled $0$. 
    The $\theta_u$ experiment corresponds to sampling $u$ balls sequentially without replacement and writing down the labels of the balls as the binary sequence $0.X_1\dotsi X_u$. 
    Instead, if the balls are chosen with replacement, then their labels form the binary sequence $0.Y_1\dotsi Y_u$. 

    Let $C$ denote the collision event that some ball is picked twice. 
    Conditioned on $C$ not happening (\ie, no ball is picked twice), the distribution of $0.Y_1\dotsi Y_u$ is identical to the distribution of $\theta_u$. 
    Denote the probability of collision by $p(C)$; the complementary, no-collision, probability by $p(\neg C)$.
    Therefore, we conclude:
    %%%
    \begin{align*}
        \expect{\exp(2\pi\imath \cdot 0.Y_1\dotsi Y_u)} 
        %%%
            &= \expect{1(\neg C)\cdot \exp(2\pi\imath \cdot 0.Y_1\dotsi Y_u)}\\
                &\qquad \qquad + \expect{1(C)\cdot \exp(2\pi\imath \cdot 0.Y_1\dotsi Y_u)}
                \tag{separating the count based on $C$ occurring or not}\\
        %%%
            &= p(\neg C) \cdot \expect{\exp(2\pi\imath \cdot \theta_u)} + \expect{1(C)\cdot \exp(2\pi\imath \cdot 0.Y_1\dotsi Y_u)}
                \tag{conditioned on no collision, it's identical to the hypergeometric experiment}\\
        %%%
            &= \expect{\exp(2\pi\imath \cdot \theta_u)} \quad -p(C)\cdot \expect{\exp(2\pi\imath \cdot \theta_u)}\\
                &\qquad\qquad + \expect{1(C)\cdot \exp(2\pi\imath \cdot 0.Y_1\dotsi Y_u)}.
    \end{align*}

    Since we are computing the expectation of only unimodular objects, we conclude that 
        \[
            \abs{ \expect{\exp(2\pi\imath \cdot \theta_u)} - \expect{\exp(2\pi\imath \cdot 0.Y_1\dotsi Y_u)} \vphantom{2^{2^2}} } \leq 2\cdot p(C).
        \]
    Now, the collision probability, by the union bound, is upper-bounded by $p(C)\leq \binom u2\cdot \frac1n$. 
    This proves \claimref{iid-u}. 
\end{proof}

%%% 
\begin{proof}[Proof of \claimref{ideal-est}]
    Note that at $j=1$
        \[
            \abs{ \left(1-p_n) + p_n\exp(2\pi\imath\cdot 2^{-j}\right)} = \abs{1-2p_n} .
        \]
    For $j>1$, we have
    \begin{align*}
        \abs{ \left(1-p_n) + p_n\exp(2\pi\imath/ 2^{j}\right)} ^2
            %%%
            &= (1-p_n)^2 + p_n^2 + 2p_n(1-p_n) \cos(2\pi/2^j)\\
            %%%
            &= 1-4p_n(1-p_n)\sin^2(\pi/2^j)\\
            %%%
            &\geq \cos^2(\pi/2^j)
    \end{align*}

    From these two bounds, we have
    \begin{align*}
        \abs{ P(u)} 
            %%%
            & \geq \abs{1-2p_n}  \cdot \cos(\pi/4) \cos(\pi/8) \dotsi \cos(\pi/2^u) \\
            %%%
            & = \abs{1-2p_n}  \cdot \frac{\cancelto1{\sin(\pi/2)}}{2^{u-1}\sin(\pi/2^u)}\\
            %%%
            & \geq (2/\pi)\cdot \abs{1-2p_n}  .
    \end{align*}
    %%%
    Since
        \(
        w_n \geq \frac n2+\sqrt n -1,
        \)
    we have
    \[
        \abs{1-2p_n}
        =\frac{\abs{2w_n-n}}{n}
        \geq\frac{2\sqrt n-2}{n}
        \geq \frac47\cdot n^{-1/2},
    \]
    for $n\geq 3$, concluding the proof of \claimref{ideal-est}. 
\end{proof}

\bibliographystyle{alphaurl}
\phantomsection % Creates a link anchor at the current vertical position
\bibliography{math}

@article{Aaronson2018,
  author     = {Aaronson, James},
  title      = {Functions with large additive energy supported on a {H}amming sphere},
  year       = {2018},
  eprinttype = {arXiv},
  eprint     = {1805.05295}
}

@article{AlonBourgain2014,
  title={{Additive patterns in multiplicative subgroups}},
  author={Alon, Noga and Bourgain, Jean},
  journal={Geometric and Functional Analysis},
  volume={24},
  number={3},
  pages={721--739},
  year={2014},
  publisher={Springer},
  doi={10.1007/s00039-014-0270-y}
}

@book{AlloucheShallit2003,
  author     = {Allouche, Jean-Paul and Shallit, Jeffrey},
  title      = {Automatic Sequences: Theory, Applications, Generalizations},
  subtitle   = {Theory, applications, generalizations},
  publisher  = {Cambridge University Press},
  address    = {Cambridge},
  year       = {2003},
  pages      = {xvi+571},
  isbn       = {0-521-82332-3; 978-0-521-82332-6},
  review     = {\MR{2081577}},
  doi        = {10.1017/CBO9780511546563},
}

@book{berman1994nonnegative,
  title={Nonnegative Matrices in the Mathematical Sciences},
  author={Berman, Abraham and Plemmons, Robert J.},
  year={1994},
  publisher={Society for Industrial and Applied Mathematics (SIAM)},
  address={Philadelphia, PA},
  isbn={978-0-89871-321-3},
  series={Classics in Applied Mathematics},
  volume={9},
  doi={10.1137/1.9781611971262}
}

@article{Chang2002,
  AUTHOR = {Chang, Mei-Chu},
   TITLE = {A polynomial bound in {F}reiman's theorem},
 JOURNAL = {Duke Math. J.},
  VOLUME = {113},
    YEAR = {2002},
  NUMBER = {3},
   PAGES = {399--419},
    ISSN = {0012-7094},
   CODEN = {DUMJAO},
    MRCLASS = {11P70 (11B13 11B25)},
   MRNUMBER = {1955143},
 MRREVIEWER = {Yuri Bilu},
     DOI = {10.1215/S0012-7094-02-11331-3}
}

@article{Cusick2026,
  author     = {Cusick, Thomas W.},
  title      = {Proof of the {T}u--{D}eng Conjecture},
  year       = {2026},
  eprinttype = {arXiv},
  eprint     = {2608.14821}
}

@article {DiosPontGreenfeldIvanisviliMadrid2023,
    AUTHOR = {de Dios Pont, Jaume and Greenfeld, Rachel and Ivanisvili, Paata and Madrid, Jos\'{e}},
     TITLE = {Additive energies on discrete cubes},
   JOURNAL = {Discrete Anal.},
  FJOURNAL = {Discrete Analysis},
      YEAR = {2023},
     PAGES = {Paper No. 13, 16},
      ISSN = {2397-3129},
   MRCLASS = {11B30 (42B05 52C22)},
       DOI = {10.19086/da.84737}
}

@article{Delsarte1973,
  AUTHOR = {Delsarte, P.},
   TITLE = {An algebraic approach to the association schemes of coding theory},
 JOURNAL = {Philips Res. Rep. Suppl.},
  NUMBER = {10},
    YEAR = {1973},
   PAGES = {vi+97},
    NOTE = {Thesis, Universit\'e Catholique de Louvain, June 1973},
    url={https://users.wpi.edu/~martin/RESEARCH/philips.pdf},
}

@article{DiaconisFulman2009,
  author    = {Diaconis, Persi and Fulman, Jason},
  title     = {Carries, Shuffling, and an Amazing Matrix},
  journal   = {The American Mathematical Monthly},
  volume    = {116},
  number    = {9},
  pages     = {788--803},
  year      = {2009},
  doi       = {10.4169/000298909X474864}
}

@article{Filmus2016,
  author    = {Filmus, Yuval},
  title     = {An Orthogonal basis for functions over a slice of the \protect{B}oolean hypercube},
  journal   = {The Electronic Journal of Combinatorics},
  volume    = {23},
  number    = {1},
  pages     = {P1.23},
  year      = {2016},
  doi       = {10.37236/4567}
}

@inproceedings{EC:FMMOS24,
  author       = {Sebastian Faust and
                  Lo{\"{\i}}c Masure and
                  Elena Micheli and
                  Maximilian Orlt and
                  Fran{\c{c}}ois{-}Xavier Standaert},
  editor       = {Marc Joye and
                  Gregor Leander},
  title        = {Connecting Leakage-Resilient Secret Sharing to Practice: Scaling Trends
                  and Physical Dependencies of Prime Field Masking},
  booktitle    = {Advances in Cryptology - {EUROCRYPT} 2024 - 43rd Annual International
                  Conference on the Theory and Applications of Cryptographic Techniques,
                  Zurich, Switzerland, May 26-30, 2024, Proceedings, Part {IV}},
  series       = {Lecture Notes in Computer Science},
  volume       = {14654},
  pages        = {316--344},
  publisher    = {Springer},
  year         = {2024},
  url          = {https://doi.org/10.1007/978-3-031-58737-5_12},
  doi          = {10.1007/978-3-031-58737-5_12},
  bibsource    = {dblp computer science bibliography, https://dblp.org}
}

@article{Gelfond1968,
  author    = {Gelfond, A. O.},
  title     = {Sur les nombres qui ont des propri{\'e}t{\'e}s additives et multiplicatives donn{\'e}es},
  journal   = {Acta Arithmetica},
  year      = {1968},
  volume    = {13},
  number    = {3},
  pages     = {259-265},
  issn      = {0065-1036},
  url       = {http://eudml.org/doc/204828},
  doi = {10.4064/AA-13-3-259-265}
}

@article{Green2025,
  author    = {Green, Ben},
  title     = {Waring's problem with restricted digits},
  journal   = {Compositio Mathematica},
  volume    = {161},
  year      = {2025},
  number    = {2},
  pages     = {341--364},
  publisher = {Cambridge University Press},
  doi   = {10.1112/S0010437X24007723},
}

@article {HeubergerKropfProdinger2015,
    AUTHOR = {Heuberger, Clemens and Kropf, Sara and Prodinger, Helmut},
     TITLE = {Output sum of transducers: limiting distribution and periodic
              fluctuation},
   JOURNAL = {Electron. J. Combin.},
  FJOURNAL = {The Electronic Journal of Combinatorics},
    VOLUME = {22},
      YEAR = {2015},
    NUMBER = {2},
     PAGES = {Paper 2.19, 53},
      ISSN = {1077-8926},
      MRCLASS = {11B85 (05A16 60F05 68Q45)},
  MRNUMBER = {3359050},
       DOI = {10.37236/5026},
       URL = {https://doi.org/10.37236/5026},
}

@article{Holte1997,
  author    = {Holte, John M.},
  title     = {Carries, combinatorics, and an amazing matrix},
  journal   = {Amer. Math. Monthly},
  volume    = {104},
  year      = {1997},
  number    = {2},
  pages     = {138--149},
  mrnumber  = {1434314},
  doi       = {10.1080/00029890.1997.11990612},
}

@article{Konyagin2001,
  author    = {Konyagin, Sergei Vladimirovich},
  title     = {Arithmetic properties of integers with missing digits: distribution in residue classes},
  journal   = {Periodica Mathematica Hungarica},
  volume    = {42},
  number    = {1--2},
  pages     = {145--162},
  year      = {2001},
  doi       = {10.1023/A:1015256809636}
}

@article{Kovac2023,
  author = {Kova{\v{c}}, Vjekoslav},
  title = {On binomial sums, additive energies, and lazy random walks},
  journal = {J. Math. Anal. Appl.},
  volume = {528},
  year = {2023},
  number = {1},
  pages = {Article 127510, 13},
  doi = {10.1016/j.jmaa.2023.127510},
  url = {https://doi.org/10.1016/j.jmaa.2023.127510}
}

@book {KonyaginShparlinski1999,
    AUTHOR = {Konyagin, Sergei V. and Shparlinski, Igor E.},
     TITLE = {{Character sums with exponential functions and their
              applications}},
    SERIES = {Cambridge Tracts in Mathematics},
    VOLUME = {136},
 PUBLISHER = {Cambridge University Press, Cambridge},
      YEAR = {1999},
     PAGES = {viii+163},
      ISBN = {0-521-64263-9},
   MRCLASS = {11L40 (11L03 11L07)},
  MRNUMBER = {1725241},
MRREVIEWER = {D.\ R.\ Heath-Brown},
       DOI = {10.1017/CBO9780511542930},
       URL = {https://doi.org/10.1017/CBO9780511542930},
}

@article{KirshnerSamorodnitsky2020,
  author  = {Kirshner, Naomi and Samorodnitsky, Alex},
  title   = {On the {$\ell_4:\ell_2$} Ratio of Functions with Restricted \protect{F}ourier Support},
  journal = {J. Combin. Theory Ser. A},
  volume  = {172},
  pages   = {105202, 25 pp.},
  year    = {2020},
  doi     = {10.1016/j.jcta.2019.105202},
  url     = {https://doi.org/10.1016/j.jcta.2019.105202}
}

@article {KaneTao2017,
    AUTHOR = {Kane, Daniel and Tao, Terence},
     TITLE = {A bound on partitioning clusters},
   JOURNAL = {Electronic Journal of Combinatorics},
  FJOURNAL = {The Electronic Journal of Combinatorics},
    VOLUME = {24},
      YEAR = {2017},
    NUMBER = {2},
     PAGES = {Paper No. 2.31, 13},
      ISSN = {1077-8926},
   MRCLASS = {11B30 (05A18 92D15)},
  MRNUMBER = {3665324},
       DOI = {10.37236/6797}
}

@article {Maynard2019,
    AUTHOR = {Maynard, James},
     TITLE = {Primes with restricted digits},
   JOURNAL = {Invent. Math.},
  FJOURNAL = {Inventiones mathematicae},
    VOLUME = {217},
      YEAR = {2019},
    NUMBER = {1},
     PAGES = {127--218},
      ISSN = {0020-9910},
      MRCLASS = {11N05 (11K36 11L07 11N36)},
   MRNUMBER = {3967468},
       DOI = {10.1007/s00222-019-00865-6},
}

@book{matousek2008invitation,
  author    = {Matou\v{s}ek, Ji\v{r}\'{\i} and Ne\v{s}et\v{r}il, Jaroslav},
  title     = {Invitation to Discrete Mathematics},
  edition   = {2},
  publisher = {Oxford University Press},
  address   = {Oxford},
  year      = {2008},
  isbn      = {978-0198570424},
  url       = {https://global.oup.com/academic/product/an-invitation-to-discrete-mathematics-9780198570424}
}

@article {MauduitPomeranceSarkozy2005,
    AUTHOR = {Mauduit, Christian and Pomerance, Carl and S{\'a}rk{\"o}zy, Andr{\'a}s},
     TITLE = {On the distribution in residue classes of integers with a fixed
              sum of digits},
   JOURNAL = {Ramanujan J.},
  FJOURNAL = {The Ramanujan Journal},
    VOLUME = {9},
      YEAR = {2005},
    NUMBER = {1-2},
     PAGES = {45--62},
      ISSN = {1382-4090},
   MRCLASS = {11N69 (11A63 11L07 11N60)},
  MRNUMBER = {2166381},
MRREVIEWER = {Jean-Marie De Koninck},
       DOI = {10.1007/s11139-005-0824-6}
}

@article{MauduitRivat2010,
  author    = {Mauduit, Christian and Rivat, Jo{\"e}l},
  title     = {Sur un probl{\`e}me de {G}elfond: la somme des chiffres des nombres premiers},
  journal   = {Annals of Mathematics},
  volume    = {171},
  year      = {2010},
  number    = {3},
  pages     = {1591--1646},
  doi       = {10.4007/annals.2010.171.1591}
}

@article{MauduitRivatSarkozy2017,
  author   = {Mauduit, Christian and Rivat, Jo{\"e}l and S{\'a}rk{\"o}zy, Andr{\'a}s},
  title    = {On the digits of sumsets},
  journal  = {Canad. J. Math.},
  volume   = {69},
  year     = {2017},
  number   = {3},
  pages    = {595--612},
  mrnumber = {3641775},
  doi      = {10.4153/CJM-2016-007-2}
}

@book {MacWilliamsSloane1977,
    AUTHOR = {MacWilliams, Florence J. and Sloane, Neil J. A.},
     TITLE = {The Theory of Error-Correcting Codes},
    SERIES = {North-Holland Mathematical Library},
    VOLUME = {16},
 PUBLISHER = {North-Holland Publishing Co., Amsterdam-New York-Oxford},
      YEAR = {1977},
     PAGES = {xxii+762},
      ISBN = {0-444-85006-2},
   MRCLASS = {94.02},
  MRNUMBER = {0465509},
MRREVIEWER = {Philippe Piret},
}

@article {MauduitSarkozy1997,
    AUTHOR = {Mauduit, Christian and S{\'a}rk{\"o}zy, Andr{\'a}s},
     TITLE = {On the arithmetic structure of the integers whose sum of digits is fixed},
   JOURNAL = {Acta Arith.},
  FJOURNAL = {Acta Arithmetica},
    VOLUME = {81},
      YEAR = {1997},
    NUMBER = {2},
     PAGES = {145-173},
      ISSN = {0065-1036},
   MRCLASS = {11N60 (11A63 11K65 11N35)},
  MRNUMBER = {1472265},
MRREVIEWER = {G. Turnwald},
       URL = {https://eudml.org/doc/207059},
       doi={10.4064/aa-81-2-145-173},
}

@article{Polyanskiy2019,
  author = {Polyanskiy, Yury},
  title = {Hypercontractivity of Spherical Averages in \protect{H}amming Space},
  journal = {SIAM Journal on Discrete Mathematics},
  volume = {33},
  number = {2},
  pages = {731--754},
  year = {2019},
  publisher = {Society for Industrial and Applied Mathematics},
  doi={10.1137/15M1046575}
}

@article{Shao2026,
  author  = {Shao, Xuancheng},
  title   = {Additive Energies of Subsets of Discrete Cubes},
  journal = {Proc. Roy. Soc. Edinburgh Sect. A},
  volume  = {156},
  number  = {3},
  pages   = {944--965},
  year    = {2026},
  doi     = {10.1017/prm.2024.126},
  url     = {https://doi.org/10.1017/prm.2024.126}
}

@book {TaoVu2006,
    AUTHOR = {Tao, Terence and Vu, Van},
     TITLE = {{Additive combinatorics}},
    SERIES = {Cambridge Studies in Advanced Mathematics},
    VOLUME = {105},
 PUBLISHER = {Cambridge University Press, Cambridge},
      YEAR = {2006},
     PAGES = {xviii+512},
      ISBN = {978-0-521-85386-6; 0-521-85386-9},
   MRCLASS = {11-02 (05-02 05D10 11B13 11P70 11P82 28D05 37A45)},
  MRNUMBER = {2289012},
MRREVIEWER = {Serge\u{\i} V. Konyagin},
       DOI = {10.1017/CBO9780511755149}
}

\end{document}